\documentclass[10pt,a4paper,bibtotoc]{scrartcl}

\usepackage[margin=3cm]{geometry}
\usepackage{afterpage}%to change margins on single pages

\usepackage[colorlinks=true, allcolors=blue]{hyperref}
\usepackage{amsmath,amsfonts,amssymb,mathtools,amsthm}
\usepackage{graphicx}

\usepackage[toc,page]{appendix}

\usepackage[caption=false, justification=justified]{subfig} %mk: use this package instead of caption

\usepackage{authblk}%for author affiliations

\usepackage[numbers,sort&compress]{natbib}
\newcommand{\R}{\mathbb{R}}

\newcommand{\N}{\mathbb{N}}

\newcommand{\EE}{\mathbb{E}}
\newcommand{\PP}{\mathbb{P}}
\newcommand{\NN}{\mathbb{N}}

\newcommand{\e}{\mathrm{e}}
\renewcommand{\i}{\mathrm{i}}
\renewcommand{\d}{\mathrm{d}}
\newcommand{\abs}[1]{{\left\vert #1 \right\vert}}

\renewcommand{\Re}{\operatorname{Re}}

\newcommand{\roundbra}[1]{\left(#1\right)}
\newcommand{\sqbra}[1]{\left[#1\right]}

\newcommand{\curlybra}[1]{\left\{#1\right\}}

\DeclareMathOperator*{\Res}{Res}
\DeclareMathOperator{\sign}{sign}

\newtheorem{theorem}{Theorem}
\newtheorem{corollary}[theorem]{Corollary}
\newtheorem{lemma}[theorem]{Lemma}
\newtheorem{proposition}[theorem]{Proposition}
\newtheorem{remark}[theorem]{Remark}
\newtheorem{definition}{Definition}
\makeatletter % metadata for hyperref
 \hypersetup{pdftitle = {On products of Gaussian random variables},
	     pdfauthor = {\v{Z}eljka Stojanac, Daniel Suess, Martin Kliesch},
	     pdfsubject = {Probability theory},
	     pdfkeywords = {Gaussian random variable,
	     				product distribution,
	     				Meijer G-function,
	     				Fox H-function,
	     				Mellin transform,
	     				Chernoff bound,
	     				moment generating function,
	     				cumulative distribution function}
	    }
\makeatother

\newcommand{\uc}{%
      Institute for Theoretical Physics,
      University of Cologne,
      Germany}

\newcommand{\ug}{
      Institute of Theoretical Physics and Astrophysics,
      University of Gda\'nsk,
      Poland}

\title{On products of Gaussian random variables}
\date{\today}

\author[1]{\v{Z}eljka Stojanac\thanks{E-mail: stojanac@thp.uni-koeln.de}}
\author[1]{Daniel Suess\thanks{E-mail: dsuess@thp.uni-koeln.de}}
\author[2]{Martin Kliesch\thanks{E-mail: science@mkliesch.eu}}
\affil[1]{\uc}
\affil[2]{\ug}
\begin{document}
%%% ========================================================

\maketitle
\thispagestyle{empty}

\begin{abstract}
Sums of independent random variables form the basis of many fundamental theorems in probability theory and statistics, and therefore, are well understood.
The related problem of characterizing products of independent random variables seems to be much more challenging.
In this work, we investigate such products of normal random variables, products of their absolute values, and products of their squares.
We compute power-log series expansions of their cumulative distribution function (CDF) based on the theory of Fox H-functions.
Numerically we show that for small arguments the CDFs are well approximated by the lowest orders of this expansion.
For the two non-negative random variables, we also compute the moment generating functions in terms of Meijer G-functions, and consequently, obtain a Chernoff bound for sums of such random variables.

\medskip
\noindent \textbf{Keywords:} 
Gaussian random variable, product distribution, Meijer G-function, Chernoff bound, moment generating function

\medskip
\noindent \textbf{AMS subject classifications:} 
60E99,
% PROBABILITY
33C60, %: Hypergeometric integrals and functions defined by them ($E$, $G$, $H$ and $I$ functions)
% DISTRIBUTION (under Statistics) 
62E15, %: Exact distribution theory
62E17 %: Approximations to distributions (nonasymptotic)
\\
\end{abstract}

\section{Introduction and motivation}
\label{sec:intro}
Compared to sums of independent random variables, our understanding of products is much less comprehensive.
Nevertheless, products of independent random variables arise naturally in many applications including channel modeling~\cite{laneman2000energy,4012470}, wireless relaying systems~\cite{4291825}, quantum physics (product measurements of product states), as well as signal processing.
Here, we are particularly motivated by a tensor sensing problem (see Ref.~\cite{SPIE_Gauss} for the basic idea).
In this problem we consider tensors
$T \in \mathbb{R}^{n_1 \times n_2 \times \cdots \times n_d}$
and wish to recover them from measurements of the form $y_i \coloneqq \langle A_i, T \rangle$ with the \emph{sensing tensors} also being of rank one,
$A_i = {a}_i^1 \otimes {a}_i^2 \otimes \cdots \otimes {a}_i^d$ with
$a_i^j \in \mathbb{R}^{n_j}$.
Additionally, we assumed that the entries of ${{a}_i^j}$ are iid
Gaussian random variables.
Applying such maps to properly normalized rank-one tensor results in a product of $d$ Gaussian random variables.

Products of independent random variables have already been studied for more than 50 years \cite{siap/springer66} but are still subject of ongoing research \cite{Products_Gaunt,Gau17,Stein_Op_Gaunt,Stoyanov2014}.
In particular, it was shown that the probability density function of a product of certain independent and identically distributed (iid) random variables from the exponential family can be written in terms of Meijer G-functions \cite{Distribution_of_Products_70}.

In this work, we characterize cumulative distribution functions (CDFs) and moment generating functions arising from products of iid normal random variables, products of their squares, and products of their absolute values.
We provide power-log series expansions of the CDFs of these distributions and demonstrate numerically that low-order truncations of these series provide tight approximations.
Moreover, we express the moment generating functions of the two latter distributions in terms of Meijer-G functions.
We state the corresponding Chernoff bounds, which bound the CDFs of sums of such products.
Finally, we find simplified estimates of the Chernoff bounds and illustrate them numerically.

\subsection{Simple bound}
We first state a simple but quite loose bound to the CDF of the product of iid squared standard Gaussian random variables.

 \begin{figure}
     \subfloat{
 \includegraphics[width=0.32\textwidth,trim={97 240 120 217}]{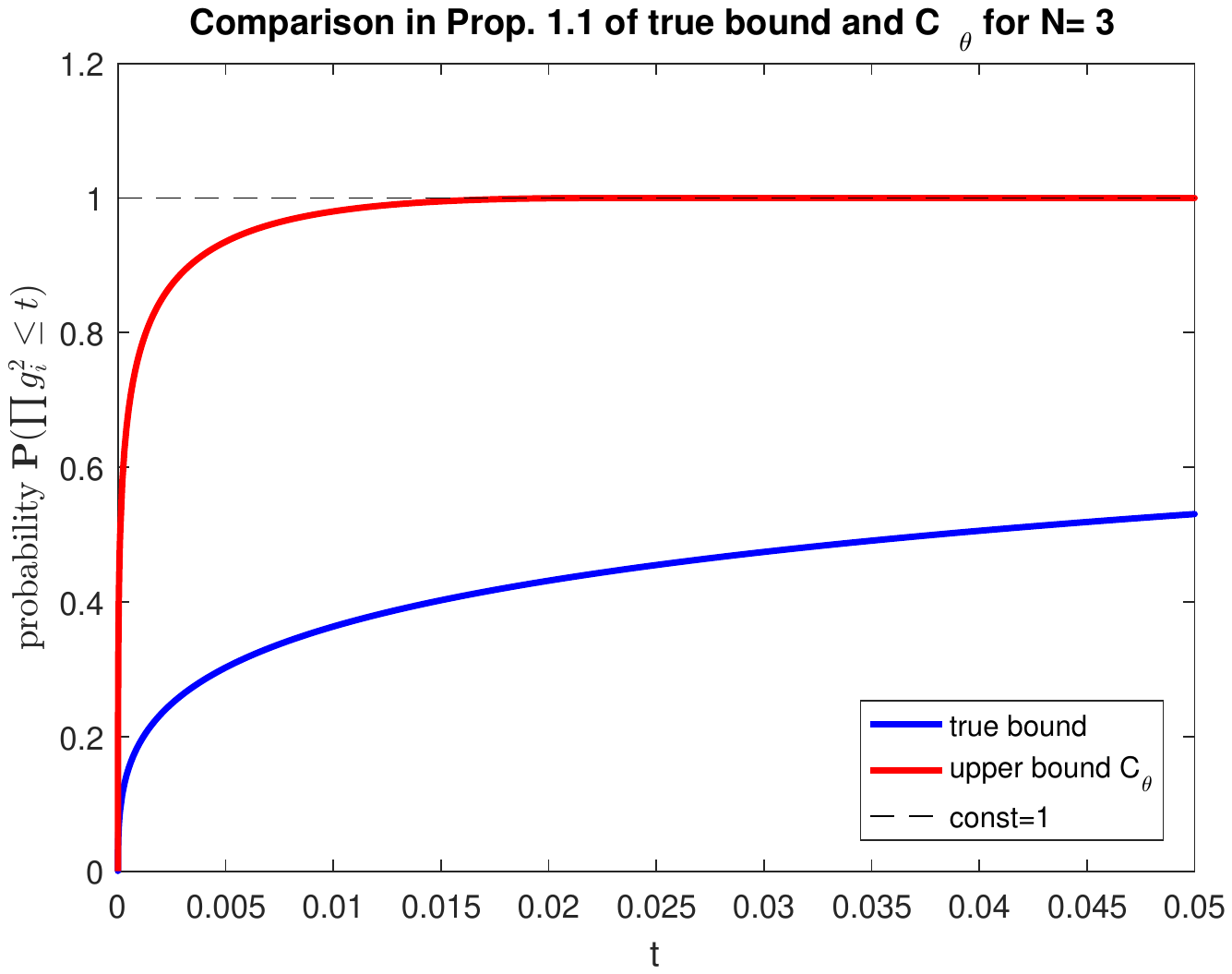}
     \hfill
 \includegraphics[width=0.32\textwidth,trim={97 240 120 217}]{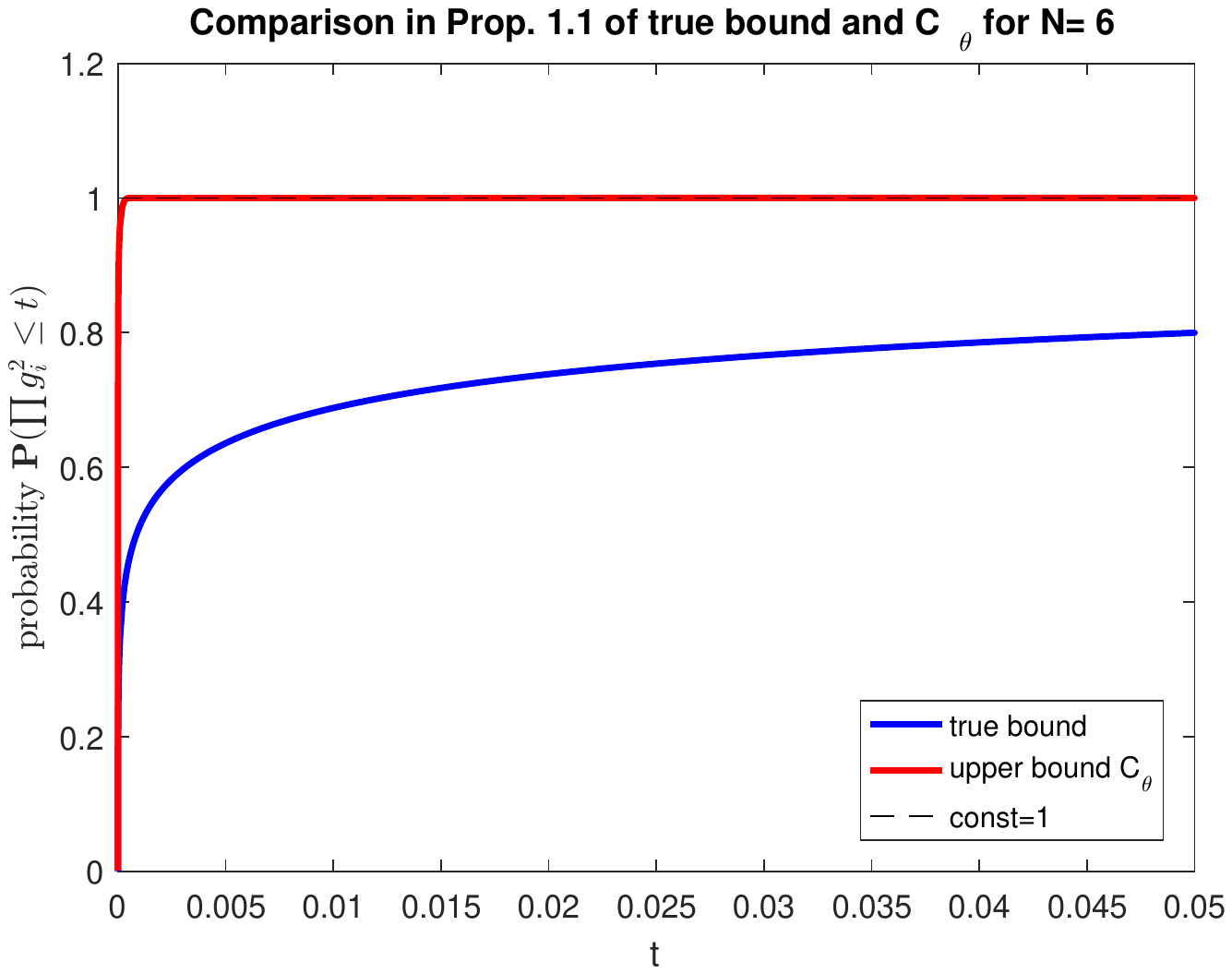}
 \hfill
\includegraphics[width=0.32\textwidth,trim={97 240 120 217}]{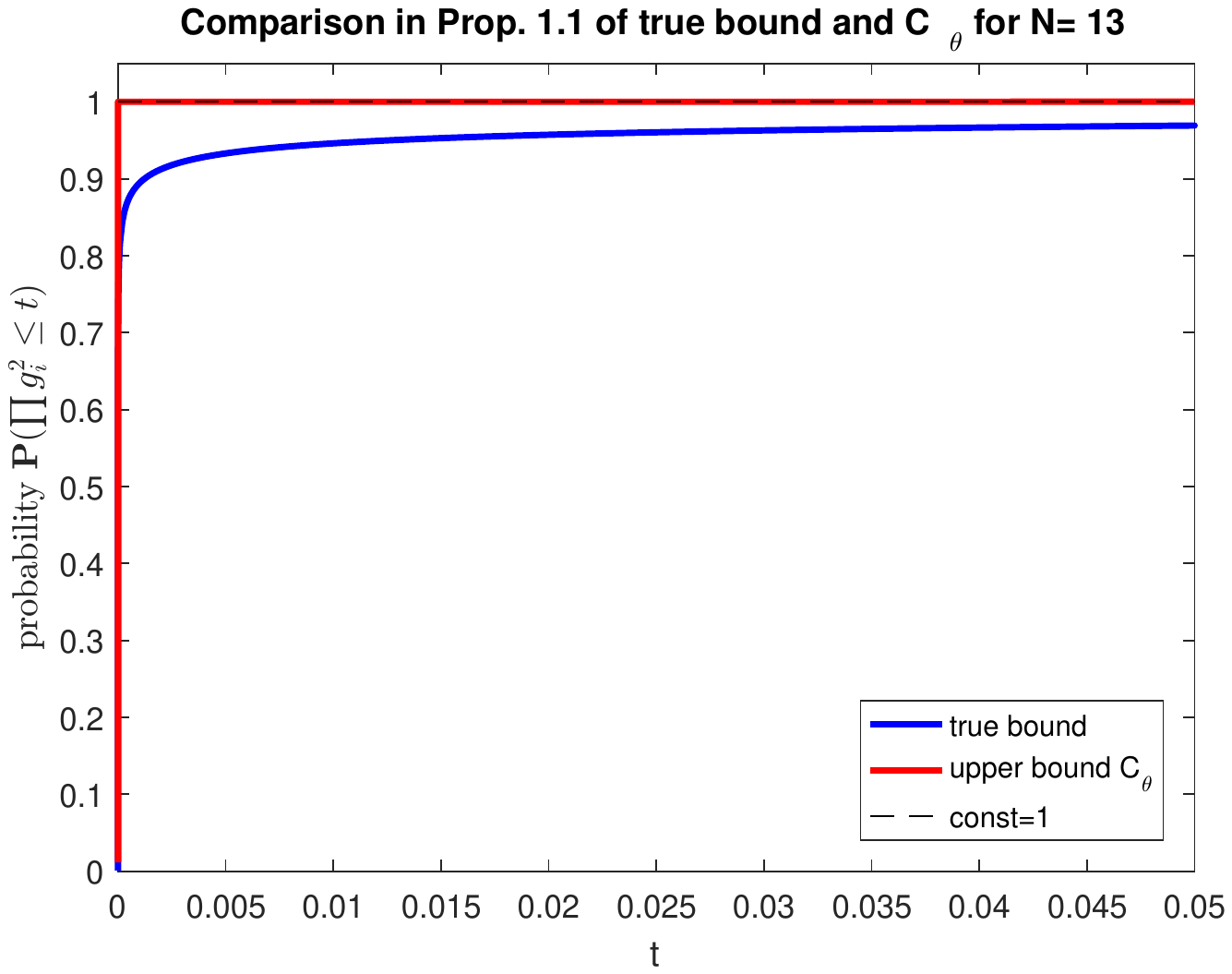}
     }
     \caption{%
       Comparison of the CDF $\mathbb{P}(\prod_{i=1}^N g_i^2 \leq t)$ (blue lines) with the bound stated in Proposition~\ref{lem:gaussian_product_measurements.one_eta} (red lines) for different values of $N$.
       Here, we chose $\theta$ for each value $t$ numerically from all possible values between $10^{-5}$ and $0.5-10^{-5}$ with step size $10^{-5}$ such that the right hand side in %Eq.~
       \eqref{eq:bound_from_chernoff} is minimal.
       \label{Fig:Chernoff}
     }
\end{figure}

\begin{proposition}
  \label{lem:gaussian_product_measurements.one_eta}
  Let $\{g_i\}_{i=1}^N$ be a set of iid standard Gaussian random variables (i.e., $g_i \sim \mathcal{N}(0, 1)$).
  Then
  \begin{equation}
    \label{eq:bound_from_chernoff}
    \PP\!\left[\prod_{i=1}^N g_i^2 \leq t \right] \le \inf_{\theta \in (0,1/2)} C_\theta^N t^\theta,
  \end{equation}
  where
  \begin{equation*}
    C_\theta = \frac{\Gamma(-\theta + \tfrac{1}{2})}{\sqrt\pi 2^\theta}>0.
  \end{equation*}
\end{proposition}

\begin{proof}
  For $\theta <\tfrac{1}{2}$ the moment generating function of $\log(g_i^2)$ is given by\footnote{%
    Throughout the article, $\log(x)$ denotes the natural logarithm of $x$.
  }
\begin{equation}
\begin{aligned}
    \EE \e^{- \theta \log g_i^2}
    &= \EE\bigl[ |g_i|^{-2 \theta} \bigr]
    = \frac{1}{\sqrt{2\pi}} \int_\mathbb{R} \abs{x}^{-2\theta} \e^{-\frac{x^2}{2}} \d x
    \\
    &= \frac{2^{-\theta}}{\sqrt{\pi}} \int_0^\infty z^{-\theta - \tfrac{1}{2}} \e^{-z} \d z
    = \frac{2^{-\theta} \, \Gamma(-\theta + \tfrac{1}{2})}{\sqrt{\pi}},
\end{aligned}
\end{equation}
  where we have used the substitution $z = \frac{x^2}{2}$.
  Note that for $\theta \geq \tfrac{1}{2}$ the integral diverges.
  Now, the proposition is a simple consequence of Chernoff bound:
\begin{equation}
  \begin{aligned}
    \PP\left( \prod_{i=1}^N g_i^2 \le t \right)
    &= \PP\left( \sum_{i=1}^N \log(g_i^2) \le \log t \right)
    \\
    &\le \inf_{\theta > 0} \e^{\theta \log t} \left(  \EE \e^{-\theta \log g_i^2}  \right)^N \\
    &= \pi^{-\tfrac{N}{2}} \, \inf_{\theta \in (0, \tfrac{1}{2})} \left(  \frac{t}{2^N}  \right)^\theta \Gamma(-\theta + \tfrac{1}{2})^N .
  \end{aligned}
\end{equation}
\end{proof}

Figure~\ref{Fig:Chernoff} shows that this bound is indeed quite loose.
This holds in particular for values of $t$ that are very small.
However, since Proposition~\ref{lem:gaussian_product_measurements.one_eta} is based on Chernoff inequality -- a tail bound on sums of random variables -- one cannot expect good results for those intermediate values.
As a matter of fact, the upper bound in 
% Eq.~
\eqref{eq:bound_from_chernoff} becomes slightly larger than the trivial bound $\PP(Y \leq t) \leq 1$ already for quite small values of $t$.
Deriving a good approximations for such values $t$ and any $N$ is in the focus of the next section.

\section{Power-log series expansion for cumulative distribution functions}
Throughout the article, we use the following notation.
Let $N \in \mathbb{N}$ and denote by $\{g_{i}\}_{i=1}^N$  a set of iid standard Gaussian random variables (i.e., $g_i \sim \mathcal N(0, 1)$, for all $i \in [N]$).
We denote the random variables considered in this work by
\begin{equation}
X\coloneqq \prod_{i=1}^N g_i,
\qquad
Y\coloneqq \prod_{i=1}^N g_i^2,
\qquad
Z\coloneqq \prod_{i=1}^N |g_i|.
\end{equation}

The probability density function of $X$ can be written in terms of  Meijer G-functions~\cite{Distribution_of_Products_70}.
We provide a similar representation for the CDF of $X$, $Y$, and $Z$ as well as the corresponding power-log series expansion using the theory developed for the more general Fox $H$-functions~\cite{Htransform}.
It is worth noting that series of Meijer-G functions have already been extensively studied. For example, in \cite{Luke} the expansions in series of Meijer G-functions as well as in series of Jacobi and Chebyshev's polynomials have been provided. 
In this work, we investigate the power-log series expansion of special instances of Meijer-G functions.
An advantage of this expansion is that it does not contain any special functions (excpet derivatives of the Gamma function at $1$). 
For the sake of completeness, we give a brief introduction on Meijer G-functions together with some of their properties in Appendix~\ref{sec:intro_meijer_g};
see Refs.~\cite{Distribution_of_Products_70, Products_Gaunt} for more details.

The following proposition is a special case of a result by Cook~\cite[p.\ 103]{Coo81} which relies on the Mellin transform.
However, for this particular case we provide an elementary proof.

\begin{proposition}[CDFs in terms of Meijer G-functions]
\label{prop:CDF}
Let $\{g_i\}_{i=1}^N $ be a set of independent identically distributed standard Gaussian random variables (i.e., $g_i \sim \mathcal{N}(0,1)$) and $X \coloneqq \prod_{i=1}^N g_i$, $Y \coloneqq \prod_{i=1}^N g_i^2$, $Z \coloneqq \prod_{i=1}^N |g_i|$ with $N \in \mathbb{N}$.
Define the function $\mathcal{G}_{\alpha}$ by
\begin{equation}\label{eq:g_alpha}
\mathcal{G}_{\alpha}(z)
\coloneqq
1 - \frac{1}{2^{\alpha}} \cdot \frac{1}{\pi^{\frac{N}{2}}}    G_{N+1, 1}^{0,N+1} \left(z \Big|  \begin{matrix} 1,1/2,\ldots,1/2 \\ 0 \end{matrix} \right).
\end{equation}
Then, for any $t> 0$,
\begin{align}
\PP\roundbra{X \leq t}&= \PP\roundbra{X \geq -t}  = \mathcal{G}_{1}\roundbra{\frac{2^N}{t^2}}
\\
\PP\roundbra{Y \leq t} &=  \mathcal{G}_{0}\roundbra{\frac{2^N}{t}}
\\
\PP\roundbra{Z \leq t} &= \mathcal{G}_{0}\roundbra{\frac{2^N}{t^2}}.
\end{align}
\end{proposition}
In Ref.~\cite{SPIE_Gauss} we have provided a proof of the above proposition by deriving first a result for the random variable $X$. The results for random variables $Y$ and $Z$ then follow trivially. For completeness, we derive a proof for a random variable $Y$ in the Appendix~\ref{App:Proofs} (see Lemma~\ref{lemma:CDFofY}).
The statements for $X$ and $Z$ are  simple consequences of this result, since for $t>0$,
\begin{align}
\PP (Z \leq t) &= \PP (Z^2 \leq t^2) =  \PP (Y \leq t^2) \, , \label{eq:relate_Y_Z}
\\\nonumber
\PP (X \leq t)
	&= \frac{1}{2}\roundbra{2\PP (X \leq t)}
	= \frac{1}{2} \roundbra{\PP (X \leq t) + 1 - \PP (X \leq -t)}
\\
	&= \frac{1}{2} \roundbra{\PP (-t \leq X \leq t) + 1}
	=
	\frac{1}{2} \roundbra{\PP (X^2 \leq t^2) + 1}
	= \frac{1}{2} \roundbra{\PP (Y \leq t^2) + 1}. \label{eq:relate_X_Y}
\end{align}

Now we are ready to present the main result of this section.
The proof is based on the theory of power-log series of Fox H-functions~\cite[Theorem 1.5]{Htransform}.

\begin{theorem}[CDFs as power-log series expansion]
\label{thm:CDF_as_mejier_g}
Let $\{g_i\}_{i=1}^N $ be a set of independent identically distributed standard Gaussian random variables (i.e., $g_i \sim \mathcal{N}(0,1)$) and $X \coloneqq \prod_{i=1}^N g_i$, $Y \coloneqq \prod_{i=1}^N g_i^2$, $Z \coloneqq \prod_{i=1}^N |g_i|$ with $N \in \mathbb{N}$.
Define the function $f_{\nu,\xi}$ by
\begin{equation}\label{eq:power-log-series}
f_{\nu, \xi}(u)
\coloneqq
\nu + \frac{1}{2^{\xi}} \cdot \frac{1}{\pi^{N/2}} \sum_{k=0}^{\infty} u^{-1/2-k} \sum_{j=0}^{N-1} H_{kj} \cdot \sqbra{\log u}^j
\end{equation}
with
\begin{align}
H_{kj}
&\coloneqq
\frac{(-1)^{Nk}}{j!} \,
\sum_{n=j}^{N-1}
	\roundbra{\frac{1}{2}+k}^{-(n-j+1)}
	\nonumber \\ \label{HkjFinal}
& \quad \cdot
	\sum_{j_1+\ldots+j_N=N-1-n} \ \prod_{t=1}^N
		\curlybra{\sum_{\ell_1+\ldots+\ell_{k+1}=j_t} \frac{\Gamma^{(\ell_{k+1})}(1)}{\ell_{k+1}!} \curlybra{\prod_{i=1}^{k-1} \roundbra{k-i+1}^{-(\ell_i+1)}}}
\end{align}
and with $j_i \in \NN_0$ and $\ell_i\in \NN_0$.
Then, for any $t> 0$,
\begin{align}
\PP\roundbra{X \leq t}&= \PP\roundbra{X \geq -t}
=
f_{1/2,1}\Bigl(\frac{2^N}{t^2}\Bigr) \, , \label{eq:X_expansion}
\\
\PP\roundbra{Y \leq t} &=
f_{0,0}\Bigl(\frac{2^N}{t}\Bigr)  \, ,\label{eq:Y_expansion}
\\
\PP\roundbra{Z \leq t} &=
f_{0,0}\Bigl(\frac{2^N}{t^2}\Bigr) \, .\label{eq:Z_expansion}
\end{align}
\end{theorem}

In Figure~\ref{Fig:ComparisonLog} we compare the CDF of $X$, $Y$, and $Z$.
Moreover, we compare the approximations obtained from truncating the power-log series~\eqref{eq:power-log-series} at low orders:
taking only the leading term into account (i.e.\ $k=0$), we already obtain a  good approximation to the CDF for relatively small values of $t$.
Furthermore, the error decreases with an increasing number of factors $N$.
Truncating the power-log series at the next highest order $k=1$ yields an error that we can only resolve in the log-error-plot as shown in the insets of Figure~\ref{Fig:ComparisonLog}.
Finding explicit error bounds is still an open problem.
The main difficulty seems to be that the series \eqref{HkjFinal}
contains large terms with alternating signs given by $\sign\bigl(\Gamma^{(\ell)}(1)\bigr) = (-1)^{\ell}$ that cancel out, so that the CDFs give indeed a value in $[0,1]$.

%%%%%%%%%%%%%%%%%%%%%%%%
% \newpage
% $\ $ \vspace{1cm} $\ $
\afterpage{%
\newgeometry{left = 2cm, right = 2cm, top = 4cm}
\begin{figure}[hp!]
  \subfloat[$X=\prod_{i=1}^N g_i$]{
    \includegraphics[width=0.32\textwidth,trim={97 240 120 217}]{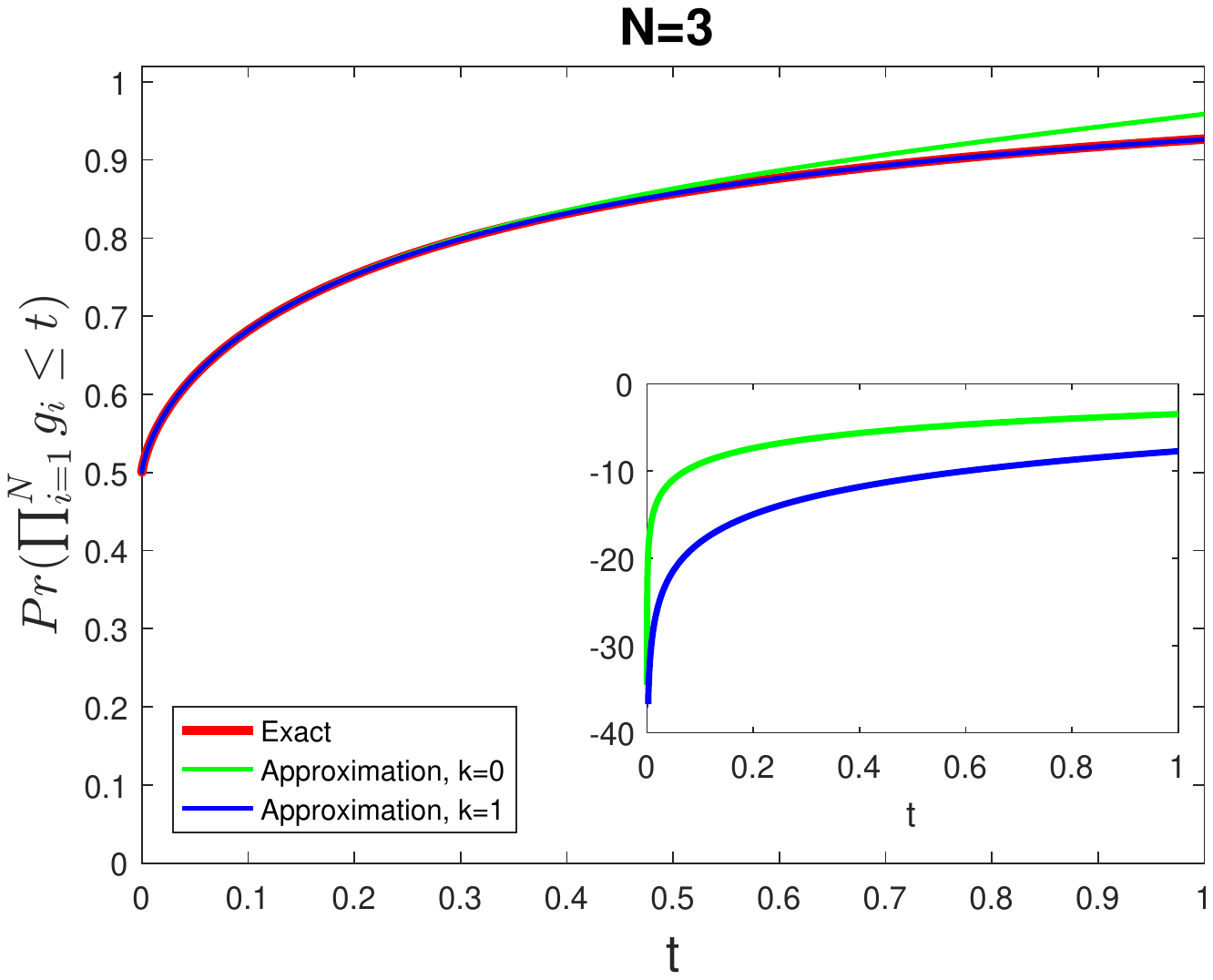}
    \hfill
    \includegraphics[width=0.32\textwidth,trim={97 240 120 217}]{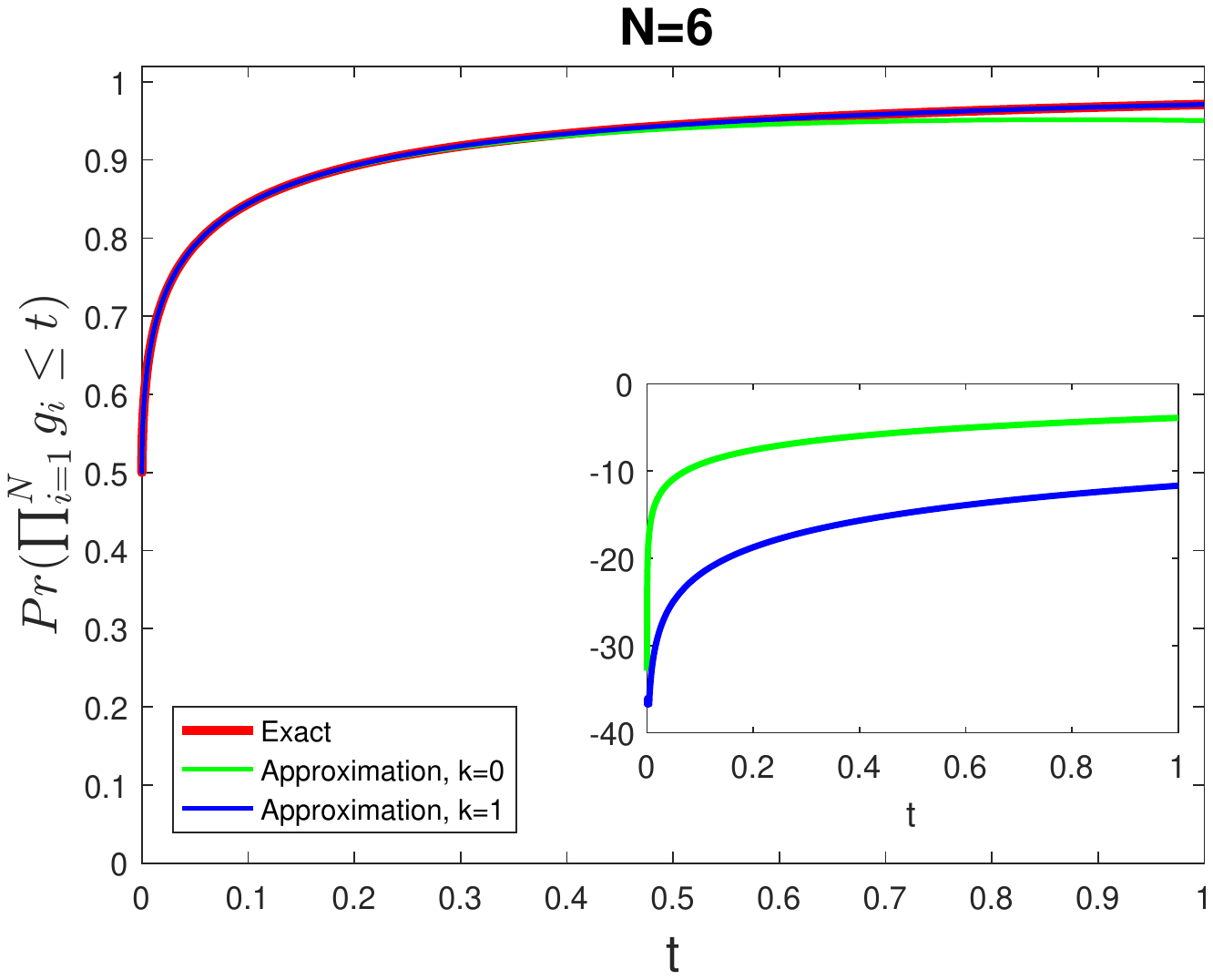}
    \hfill
    \includegraphics[width=0.32\textwidth,trim={97 240 120 217}]{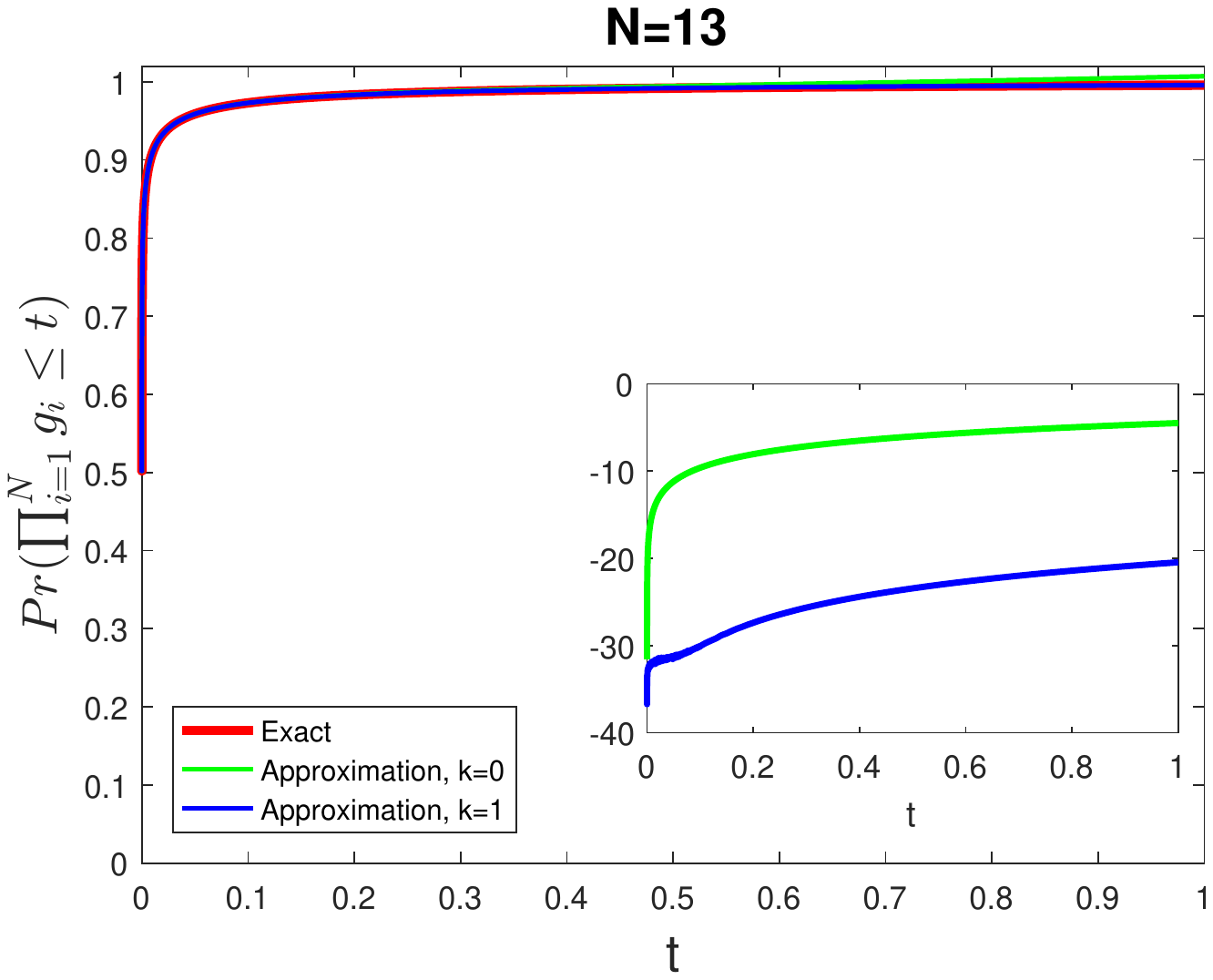}
  } \\
  \subfloat[$Y=\prod_{i=1}^N g_i^2$]{ %
    \includegraphics[width=0.32\textwidth,trim={97 240 120 217}]{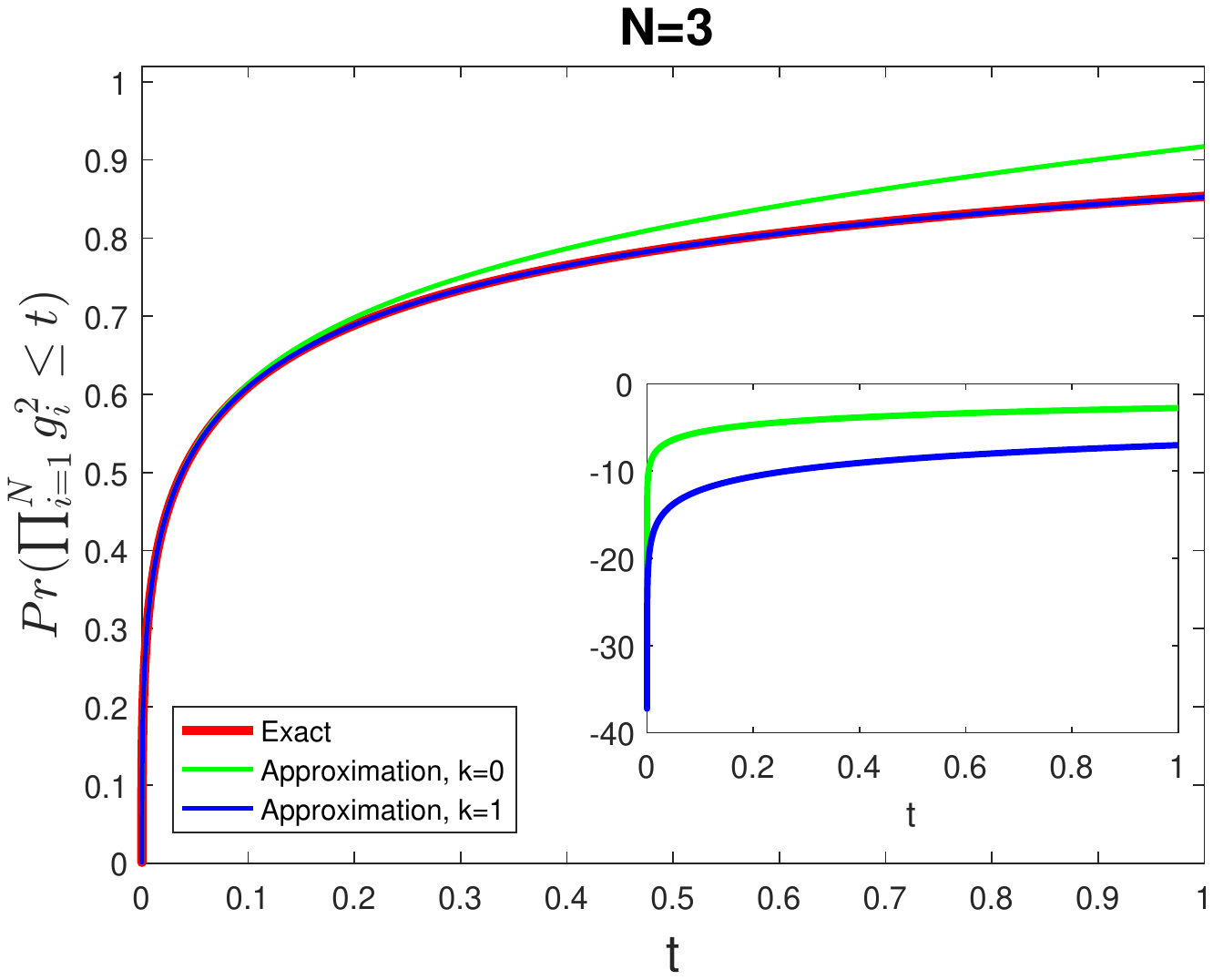}
    \hfill
    \includegraphics[width=0.32\textwidth,trim={97 240 120 217}]{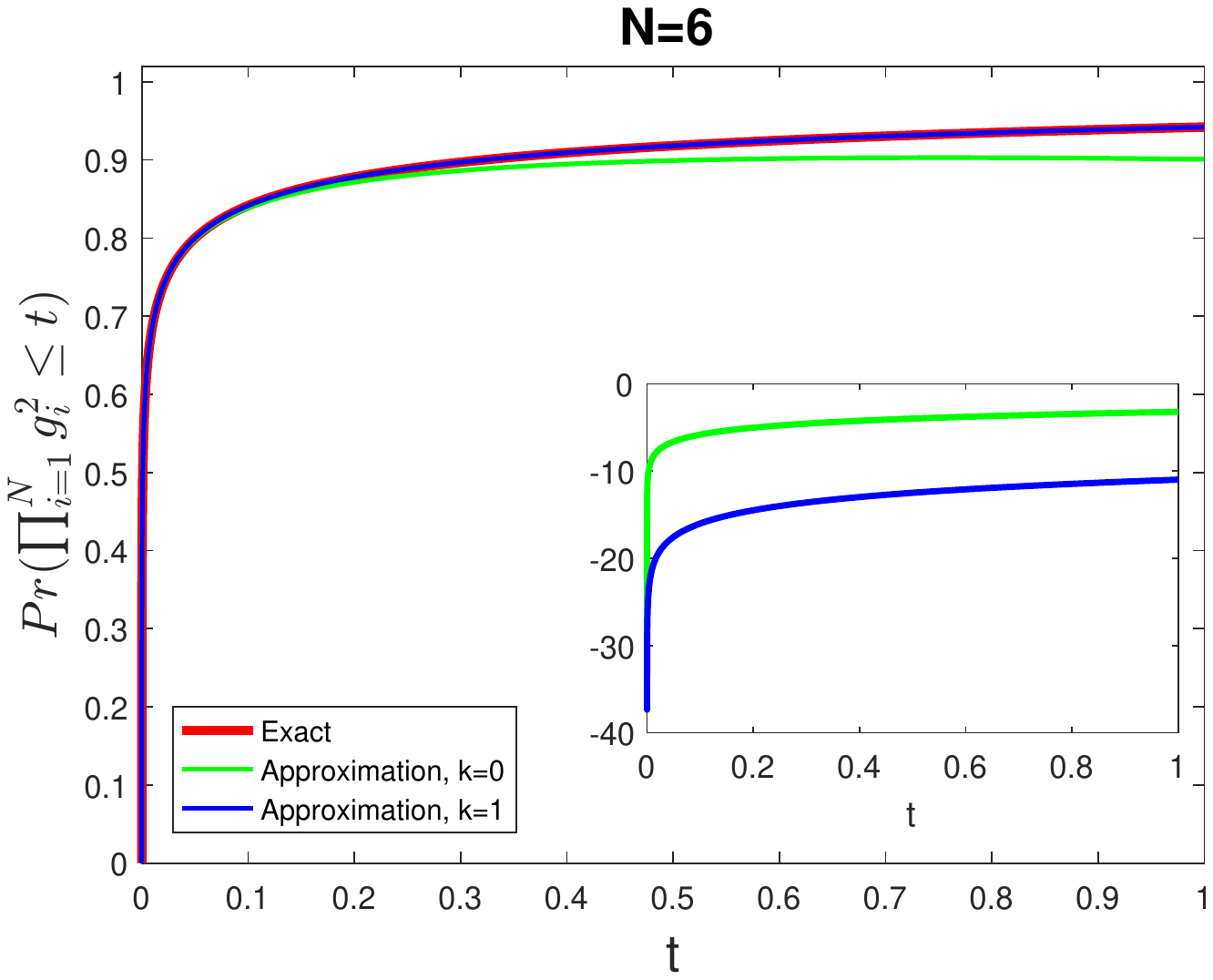}
    \hfill
    \includegraphics[width=0.32\textwidth,trim={97 240 120 217}]{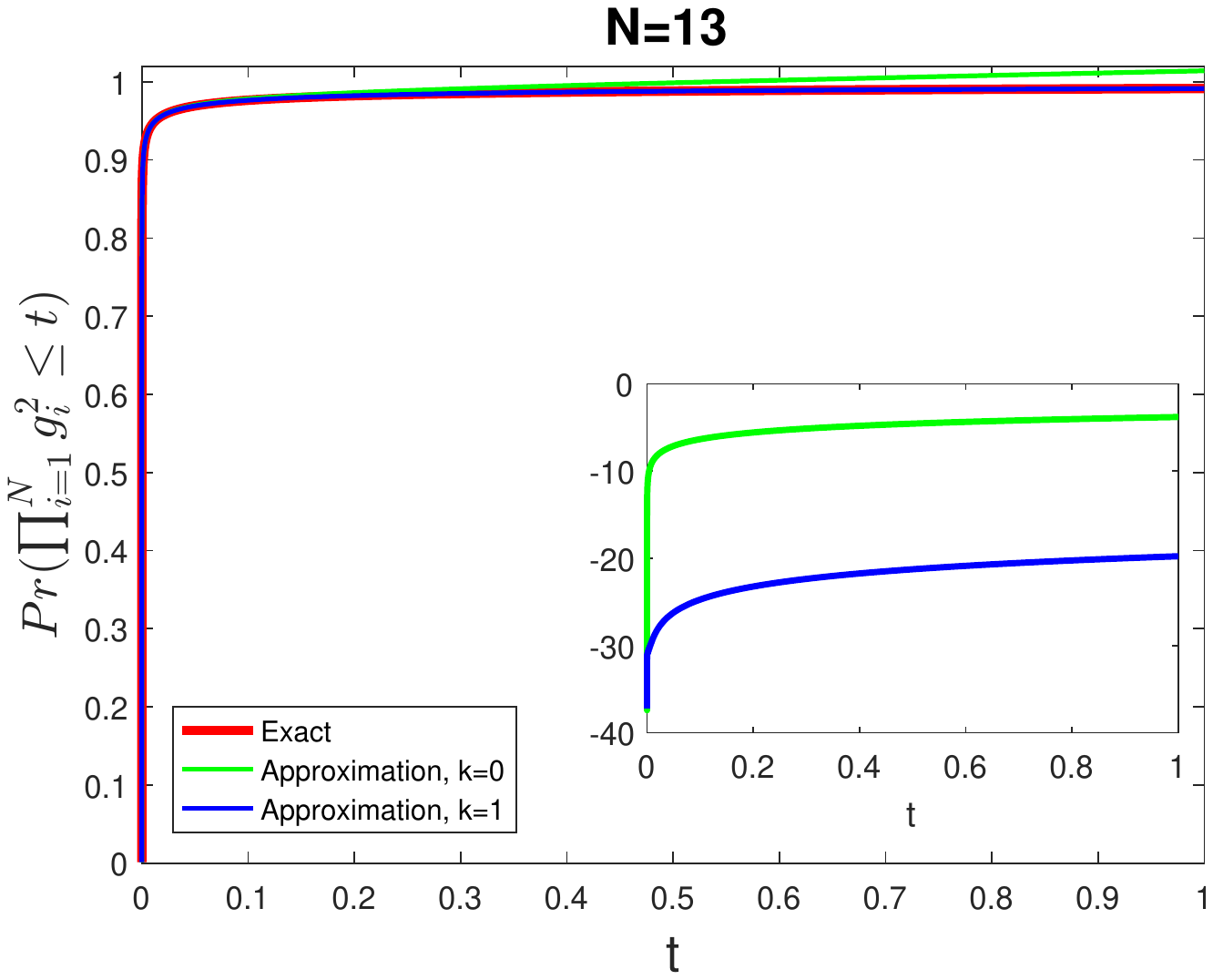}
  }\\
  \subfloat[$Z=\prod_{i=1}^N |g_i|$] {
    \includegraphics[width=0.32\textwidth,trim={97 240 120 217}]{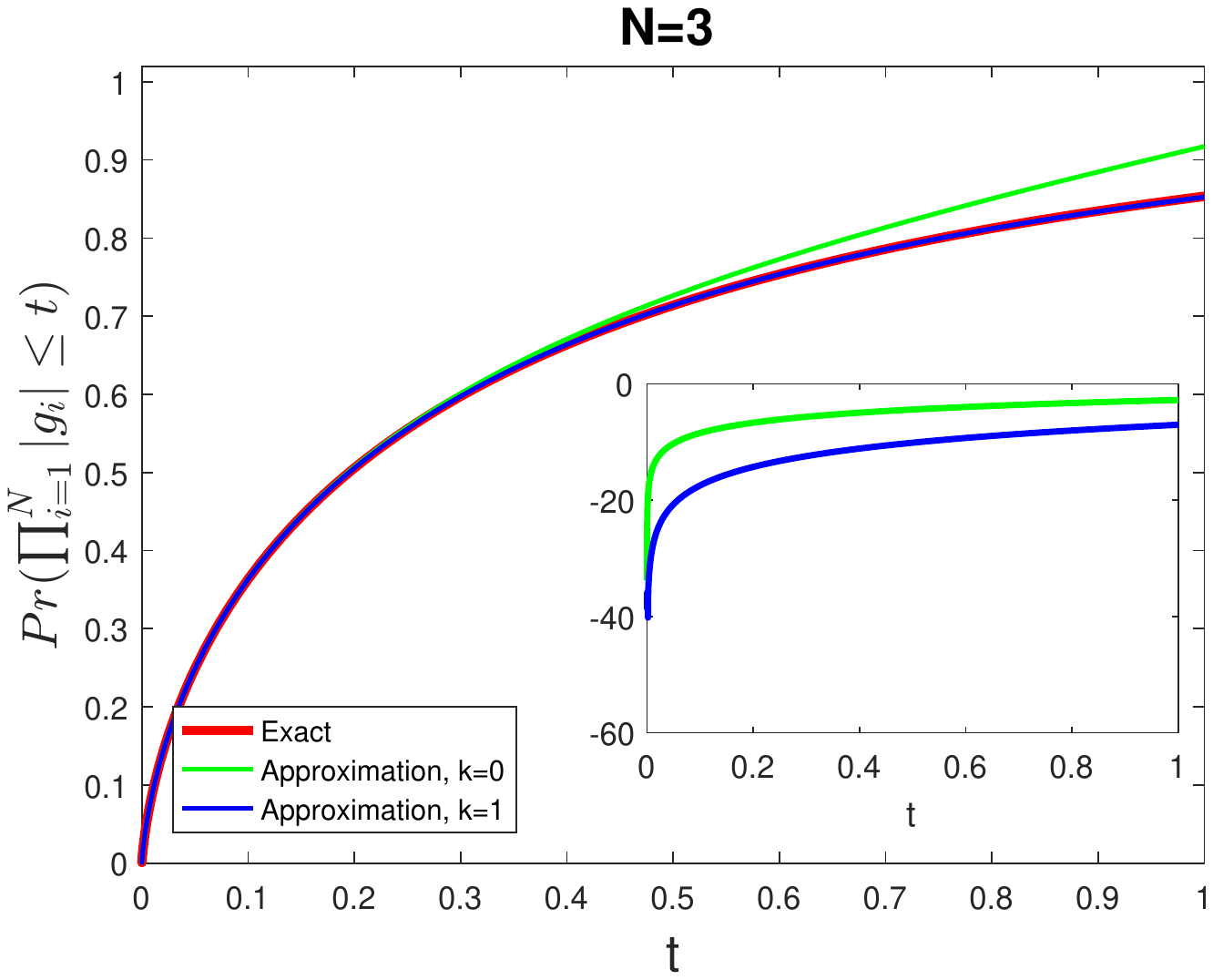}
    \hfill
    \includegraphics[width=0.32\textwidth,trim={97 240 120 217}]{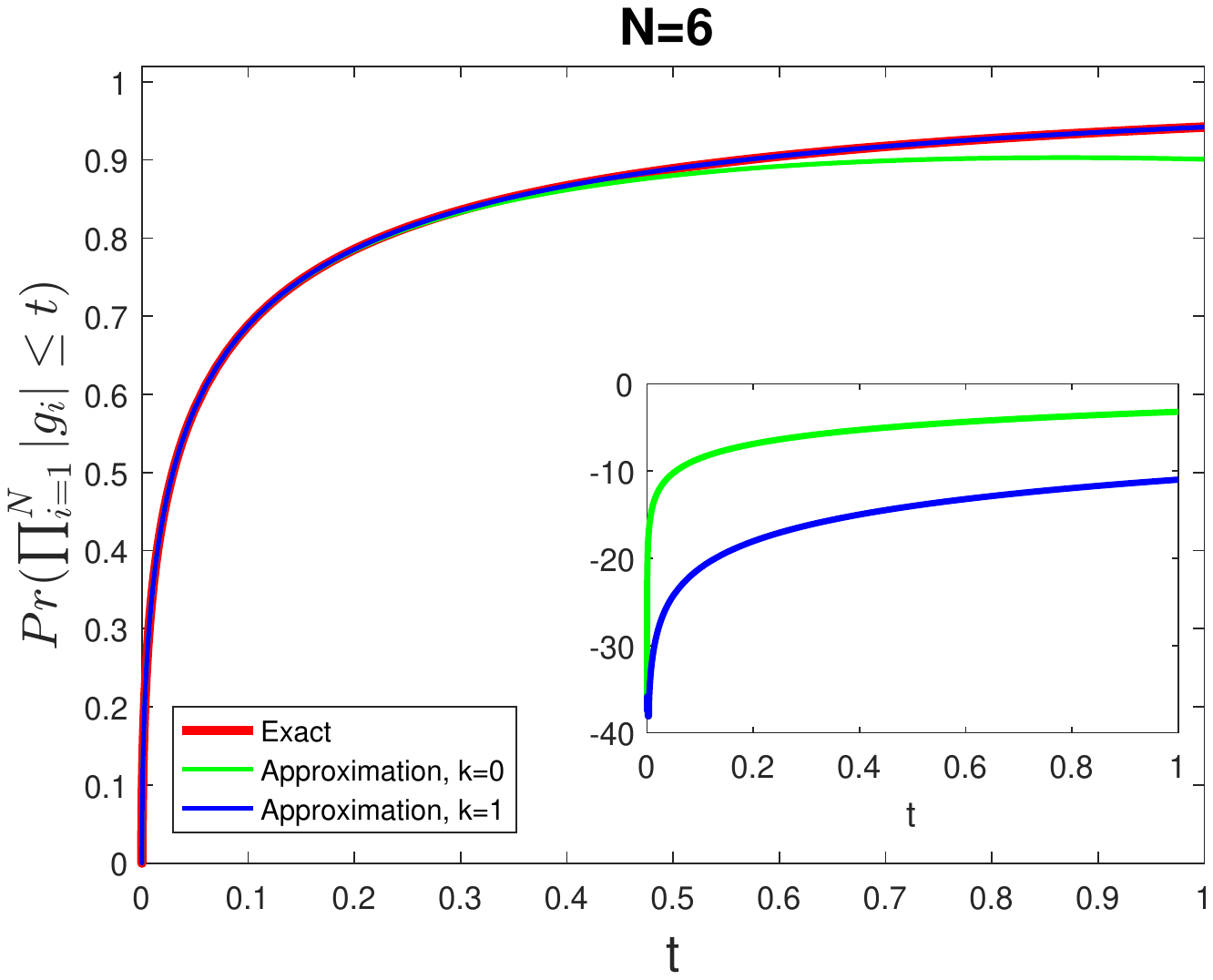}
    \hfill
    \includegraphics[width=0.32\textwidth,trim={97 240 120 217}]{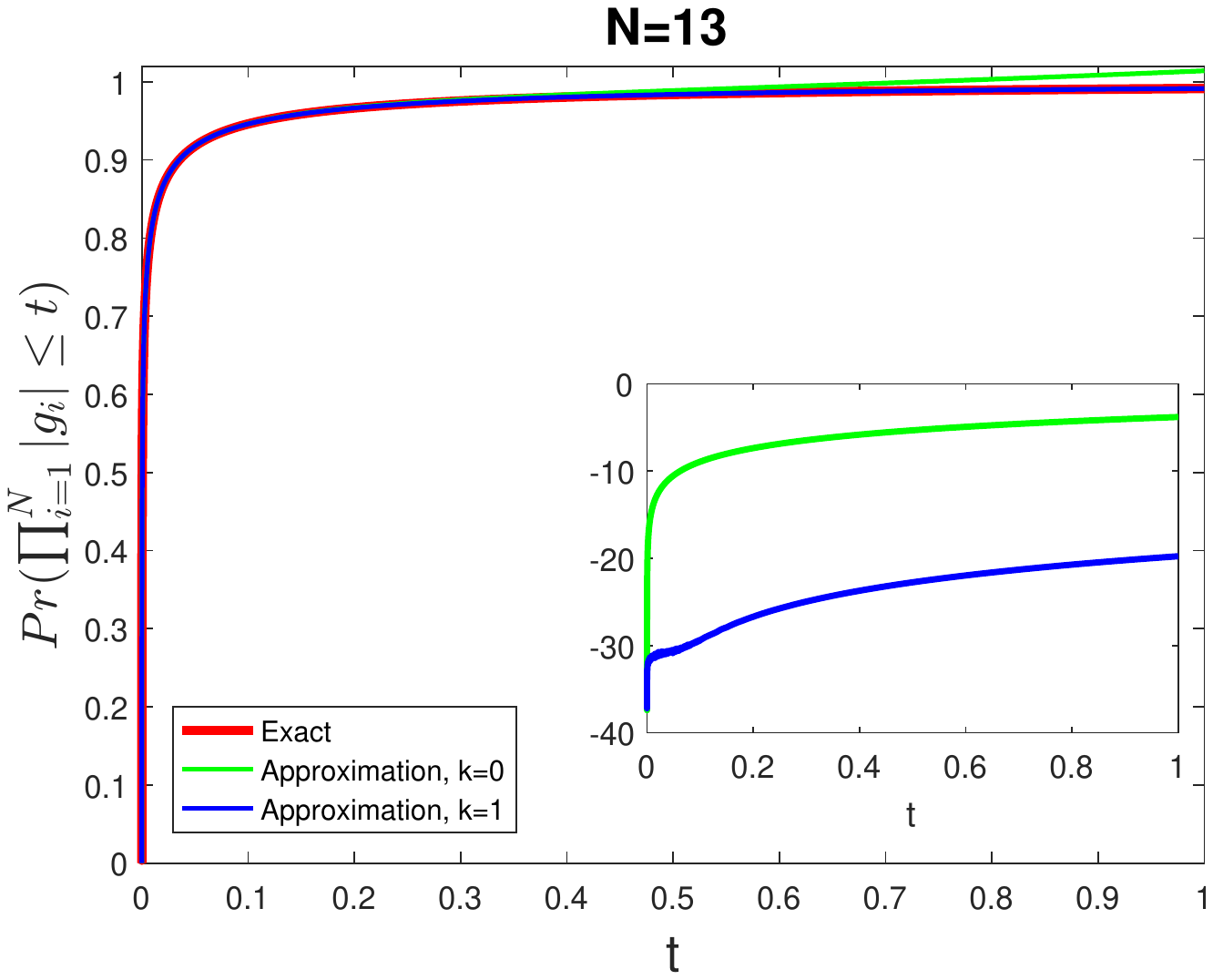}
  }\\[1em]
  \caption{%
    The plots of the  CDFs (red lines) of the random variables $X$, $Y$, and $Z$ with their power-log series \eqref{eq:power-log-series}
    truncated at $k=0$ (green lines) and $k=1$ (blue lines),
    where $\{g_i\}_{i=1}^N$ is a set of iid standard Gaussian random variables.
    Each inset shows the approximation error of the truncations on a logarithmic scale
    for $k=0$ (green line) and $k=1$ (blue line).
    \label{Fig:ComparisonLog}
  }
\end{figure}
\clearpage
\restoregeometry
} % end of \afterpage{...} material
%
% \newpage
To prove the above result, it is enough to obtain the power-log series  for the Meijer G-function from %Eq.~
\eqref{eq:g_alpha},
\begin{equation}\label{MeijerGmain}
G_{N+1,1}^{0,N+1} \left(z \Big| \begin{matrix} 1, 1/2, \ldots, 1/2 \\ 0 \end{matrix} \right) \, ,
\end{equation}
keeping in mind that in the case of random variables $X$ and $Z$ we have $z=\frac{2^N}{t^2}$, and in the case of random variable $Y$ we have $z=\frac{2^N}{t}$ with $t>0$. Then Theorem~\ref{thm:CDF_as_mejier_g} is a direct consequence of Proposition~\ref{prop:CDF}  and Theorem~\ref{thm:MeijerGseries}.
\begin{theorem}[Power-log expansion of the relevant Meijer G-function]
\label{thm:MeijerGseries}
 For $z>0$,
\begin{equation}
G_{N+1,1}^{0,N+1} \left(z \Big| \begin{matrix} 1, 1/2, \ldots, 1/2 \\ 0 \end{matrix} \right) = \pi^{\frac{N}{2}} - \sum_{k=0}^{\infty} z^{-1/2-k} \sum_{j=0}^{N-1} H_{kj} \cdot \sqbra{\log z}^j
\end{equation}
with $H_{kj}$ defined in %Eq.~
\eqref{HkjFinal}.
\end{theorem}

\begin{proof}
By~\cite[Theorem 1.2 (iii)]{Htransform}, the Meijer G-function~\eqref{MeijerGmain} is an analytic function of $z$ in the sector $\abs{\arg z}<\frac{N\pi}{2}$.
This condition is satisfied for $N \geq 3$ since $\arg z=0$ for $z \in \R_0^+$ and $\arg z=\pi$, for $z \in \R^-$.
Therefore, we have
\begin{equation}
  G_{N+1,1}^{0,N+1} \left(z \Big| \begin{matrix} 1, 1/2, \ldots, 1/2 \\ 0 \end{matrix} \right)= -\sum_{i=1}^{N+1}\sum_{k=0}^{\infty} \Res_{s=a_{ik}} \sqbra{\mathcal{H}_{N+1,1}^{0,N+1}(s) z^{-s}},
\end{equation}
where  $\mathcal{H}_{N+1,0}^{0,N+1}$ is given by %Eq.~
\eqref{Hmnpq} in the Supplemental Materials
with $m=0$, $n=p=N+1$, $q=1$, $a_1=1$, $a_2=\ldots=a_{N+1}=1/2$, and $b_1=0$, i.e.\
\begin{equation}
\mathcal{H}_{N+1,1}^{0,N+1} (s)  =  \frac{\prod_{i=1}^{N+1} \Gamma(1-a_i-s)}{\Gamma(1-b_1-s)}
\end{equation}
and $a_{ik}=1-a_i+k$ with $k \in \NN_0$ denote the poles of $\Gamma(1-a_i-s)$.
In our scenario we have simple poles $a_{1k}=k$ and  poles $a_{2k}=1/2+k$ of order $N$ with $k \in \N_0$.
Therefore,
\begin{align}
  G_{N+1,1}^{0,N+1} \left(z \Big| \begin{matrix} 1, 1/2, \ldots, 1/2 \\ 0 \end{matrix} \right)= \sum_{k=0}^{\infty} & -\Res_{s=a_{1k}} \sqbra{\mathcal{H}_{N+1,1}^{0,N+1}(s) z^{-s}}  \nonumber\\
  & + \sum_{k=0}^{\infty} -\Res_{s=a_{2k}} \sqbra{\mathcal{H}_{N+1,1}^{0,N+1}(s) z^{-s}} \label{MeijerGinRes}
\end{align}
with
\begin{align}
\mathcal{H}_{N+1,1}^{0,N+1} (s) & =  \frac{\prod_{i=1}^{N+1} \Gamma(1-a_i-s)}{\Gamma(1-b_1-s)} =  \frac{\Gamma(1-1-s) \Gamma^N(1-1/2-s)}{\Gamma(1-0-s)}
\nonumber
\\
& = \frac{\Gamma(-s) \Gamma^N(1/2-s)}{\Gamma(1-s)} = \frac{\Gamma(-s) \Gamma^N(1/2-s)}{-s \cdot \Gamma(-s)} = - \frac{ \Gamma^N(1/2-s)}{s} ,
\label{mathcalH}
\end{align}
where we have used that $\Gamma(1-s)=-s \cdot \Gamma(-s)$.
For the simple poles $a_{1k}$ we have~\cite[(1.4.10)]{Htransform}\footnote{%
  Note that there is a minus sign missing in~\cite[(1.4.10)]{Htransform} due to a typo.
}
\begin{equation} \label{eq:minus_Res}
 -\Res_{s=a_{1k}} \sqbra{\mathcal{H}_{N+1,1}^{0,N+1}(s) z^{-s}} = h_{1k} z^{-a_{1k}},
 \end{equation}
 where the constants $h_{1k}$ are
 \begin{equation}
 h_{1k}=\lim_{s \rightarrow a_{1k}} \sqbra{-(s-a_{1k}) \mathcal{H}_{N+1,1}^{0,N+1}(s)}.
 \end{equation}
For $k \in \N$ using %Eq.~
\eqref{mathcalH} we have
\begin{equation}
\begin{aligned}
 h_{1k}&=\lim_{s \rightarrow k} \sqbra{-(s-k) \mathcal{H}_{N+1,1}^{0,N+1}(s)} =  \lim_{s \rightarrow k} \sqbra{(s-k) \frac{ \Gamma^N(1/2-s)}{s} } \\
 & = \lim_{s \rightarrow k} \sqbra{\roundbra{1-\frac{k}{s}} \Gamma^N(1/2-s) } = 0. \label{h1k}
\end{aligned}
\end{equation}
For $k=0$ we have
\begin{equation}\label{h10}
 h_{10}=\lim_{s \rightarrow 0} \sqbra{-s \mathcal{H}_{N+1,1}^{0,N+1}(s)} =  \lim_{s \rightarrow 0} \sqbra{s \frac{ \Gamma^N(1/2-s)}{s} } = \lim_{s \rightarrow 0} \sqbra{\Gamma^N(1/2-s) } = \pi^{\frac{N}{2}}.
\end{equation}
Combining %Eqs.~
\eqref{h1k}, \eqref{h10}, and \eqref{eq:minus_Res} with %Eq.~
\eqref{MeijerGinRes} we obtain
\begin{equation}
G_{N+1,1}^{0,N+1} \left(z \Big| \begin{matrix} 1, 1/2, \ldots 1/2 \\ 0 \end{matrix} \right)= \pi^{\frac{N}{2}} - \sum_{k=0}^{\infty}\Res_{s=a_{2k}} \sqbra{\mathcal{H}_{N+1,1}^{0,N+1}(s) z^{-s}},
\end{equation}
where $a_{2k}=1/2+k$, for all $k \in \N_0$.
Having %Eq.~
\eqref{MeijerGinRes} in mind, the result now follows from the following lemma.
\end{proof}

\begin{lemma} [Residues of $\Gamma^N$]
\label{lem:ResNonSimple}
For $k \in \NN_0$ and $a_{2k}=1/2+k$
\begin{equation}
\Res_{s=a_{2k}} \sqbra{\mathcal{H}_{N+1,1}^{0,N+1}(s) z^{-s}} = z^{-1/2-k} \sum_{j=0}^{N-1} H_{kj} \cdot \sqbra{\log z}^j
\end{equation}
where $H_{kj}$ are defined in %Eq.~
\eqref{HkjFinal}.
\end{lemma}

In the following, we denote the $n$-th derivative of a function $f$ by $f^{(n)}$ or $[f]^{(n)}$.
Furthermore, the $n$-th derivative of product of functions $f$ and $g$ is denoted by $[f\cdot g]^{(n)}$.

\begin{proof}
Since $a_{2k}=1/2+k$ is an $N$-th order pole of $\Gamma(1/2-s)$ for all $k \in \N_0$, also the integrand $\mathcal{H}_{N+1,1}^{0,N+1}(s)= -s^{-1} \Gamma^N(1/2-s)$ has a pole $a_{2k}$ of order $N$ for all $k \in \N_0$.
Thus
\begin{align}
& \Res_{s=1/2+k} \sqbra{\mathcal{H}_{N+1,1}^{0,N+1}(s)z^{-s}} \nonumber \\
& = \frac{1}{(N-1)!} \lim_{s \rightarrow 1/2+k} \sqbra{(s-1/2-k)^N \mathcal{H}_{N+1,1}^{0,N+1}(s) z^{-s}}^{(N-1)}
\nonumber
\\
& = \frac{1}{(N-1)!} \lim_{s \rightarrow 1/2+k} \sqbra{(s-1/2-k)^N \cdot \roundbra{-s^{-1} \Gamma^N(1/2-s)} z^{-s}}^{(N-1)}
\nonumber
\\
& = \frac{1}{(N-1)!} \lim_{s \rightarrow 1/2+k} \sqbra{\roundbra{\roundbra{s-1/2-k}\Gamma\roundbra{1/2-s}}^N \cdot (-s^{-1}) z^{-s}}^{(N-1)}
\nonumber
\\
& = \frac{1}{(N-1)!} \lim_{s \rightarrow 1/2+k} \sqbra{\mathcal{H}_1(s) \cdot \mathcal{H}_2(s) z^{-s}}^{(N-1)}, \label{eq:residualder}
\end{align}
where
\begin{align}
\mathcal{H}_1(s) & \coloneqq \roundbra{\roundbra{s-1/2-k}\Gamma\roundbra{1/2-s}}^N \label{eq:defH1}\\
\mathcal{H}_2(s) & \coloneqq -s^{-1}.
\end{align}
Using the Leibniz rule and that $\frac{\partial^i}{\partial^i s}\sqbra{z^{-s}} = (-1)^i z^{-s} \sqbra{\log z}^i$ we can expand the expression in the limit in \eqref{eq:residualder}, i.e.,
\begin{align}
& \sqbra{\mathcal{H}_1(s) \cdot \mathcal{H}_2(s) z^{-s}}^{(N-1)} \nonumber\\
& = \sum_{n=0}^{N-1} \binom{N-1}{n} \sqbra{\mathcal{H}_1(s)}^{(N-1-n)} \sqbra{\mathcal{H}_2(s)z^{-s}}^{(n)}
\nonumber
\\
& = \sum_{n=0}^{N-1} \binom{N-1}{n} \sqbra{\mathcal{H}_1(s)}^{(N-1-n)} \sum_{j=0}^n \binom{n}{j} (-1)^j \sqbra{\mathcal{H}_2(s)}^{(n-j)}z^{-s} \sqbra{\log z}^j
\nonumber
\\
& =z^{-s} \sum_{j=0}^{N-1} \curlybra{ \sum_{n=j}^{N-1}  (-1)^j \binom{N-1}{n}  \binom{n}{j} \sqbra{\mathcal{H}_1(s)}^{(N-1-n)}  \sqbra{\mathcal{H}_2(s)}^{(n-j)}} \sqbra{\log z}^j . \label{eq:limitRes}
\end{align}

Plugging \eqref{eq:limitRes} into \eqref{eq:residualder} we obtain
\begin{equation}\label{eq:Resdevelop}
\begin{aligned}
&\Res_{s=1/2+k} \sqbra{\mathcal{H}_{N+1,1}^{0,N+1}(s)z^{-s}} 
\\
& = \frac{1}{(N-1)!} \lim_{s \rightarrow 1/2+k} z^{-s} 
\\
&\quad \quad  \cdot \sum_{j=0}^{N-1} \left\{ \sum_{n=j}^{N-1}  (-1)^j \binom{N-1}{n}  \binom{n}{j} %\right. \nonumber \\
%& \quad \quad \quad \quad \quad \quad \quad \cdot \left.
 \sqbra{\mathcal{H}_1(s)}^{(N-1-n)}  \sqbra{\mathcal{H}_2(s)}^{(n-j)}\right\} \sqbra{\log z}^j
\\
& = z^{-1/2-k} \sum_{j=0}^{N-1}  \left\{\frac{1}{(N-1)!} \sum_{n=j}^{N-1}  (-1)^j \binom{N-1}{n}  \binom{n}{j} \right. 
\\
& \quad \quad \quad \quad \quad \quad \quad \cdot \left. \lim_{s \rightarrow 1/2+k}  \sqbra{\mathcal{H}_1(s)}^{(N-1-n)}  \lim_{s \rightarrow 1/2+k} \sqbra{\mathcal{H}_2(s)}^{(n-j)}\right\} \sqbra{\log z}^j
\\ 
&= z^{-1/2-k} \sum_{j=0}^{N-1}  H_{kj} \cdot \sqbra{\log z}^j,
\end{aligned}
\end{equation}
where
\begin{equation}\label{eq:Hkj1}
H_{kj} \coloneqq \frac{(-1)^j}{(N-1)!} \sum_{n=j}^{N-1}  \binom{N-1}{n}  \binom{n}{j} \lim_{s \rightarrow 1/2+k}  \sqbra{\mathcal{H}_1(s)}^{(N-1-n)}  \lim_{s \rightarrow 1/2+k} \sqbra{\mathcal{H}_2(s)}^{(n-j)}.
\end{equation}
It is easy to see (proof by induction) that for $\mathcal{H}_2(s)=-s^{-1}$ and for $\ell \in \N_0$ it holds that
\begin{align}
 \mathcal{H}_2^{(\ell)}(s) & = (-1)^{\ell+1} \ell! s^{-\ell-1}  
 \intertext{and} 
 \lim_{s \rightarrow 1/2+k} \sqbra{\mathcal{H}_2(s)}^{(\ell)} & = (-1)^{\ell+1} \ell! (1/2+k)^{-\ell-1}\quad \text{for all } k \in \N_0.
\end{align}
Plugging this into \eqref{eq:Hkj1} and simplifying leads to
\begin{equation}
H_{kj}=  \sum_{n=j}^{N-1}  (-1)^{n+1} \frac{1}{j!(N-1-n)!} \roundbra{1/2+k}^{-(n-j+1)} \lim_{s \rightarrow 1/2+k}  \sqbra{\mathcal{H}_1(s)}^{(N-1-n)} .
\end{equation}
Comparing the above expression with \eqref{HkjFinal} it is enough to show that
\begin{align}
\lim_{s \rightarrow 1/2+k} &\sqbra{\mathcal{H}_1(s)}^{(N-1-n)}  = (-1)^{Nk-n-1} (N-1-n)! 
\nonumber \\
& \cdot \sum_{j_1+\ldots+j_N=N-1-n} \ \prod_{t=1}^N
		\curlybra{\sum_{\ell_1+\ldots+\ell_{k+1}=j_t} \frac{\Gamma^{(\ell_{k+1})}(1)}{\ell_{k+1}!} \curlybra{\prod_{i=1}^{k-1} \roundbra{k-i+1}^{-(\ell_i+1)}}} , \label{eq:limH1N-1-n}
\end{align}
which is a direct consequence of the following result.
\end{proof}
\begin{lemma}\label{lem:H1s}
For $\mathcal H_1$ defined in %Eq.~
\eqref{eq:defH1} and $k \in \NN_0$
\begin{align}
%&\phantom{=.}
\lim_{s \rightarrow 1/2+k}& \sqbra{\mathcal{H}_1(s)}^{(j)} = (-1)^{N(k-1)+j} j! 
\nonumber \\ \label{eq:lem:H1s}
 %= (-1)^{N(k-1)+j} j! %\,
& \quad \quad \cdot \sum_{j_1+\ldots+j_N=j} \ \prod_{t=1}^N
		\curlybra{\sum_{\ell_1+\ldots+\ell_{k+1}=j_t} \frac{\Gamma^{(\ell_{k+1})}(1)}{\ell_{k+1}!} \curlybra{\prod_{i=1}^{k-1} \roundbra{k-i+1}^{-(\ell_i+1)}}}.
\end{align}
\end{lemma}

To prove this result we  use the following lemma whose proof can be found in Appendix~\ref{App:Proofs}.

\begin{lemma}\label{lemma:limfder}
Define $f(s) \coloneqq (s-1/2-k) \Gamma(1/2-s)$. Then for every $j \in \N$ and $k \in \N_0$ it holds that
\begin{equation}
\lim_{s \rightarrow 1/2+k}f^{(j)}(s)=(-1)^{k+j-1} j! \roundbra{\sum_{\ell_1+\ell_2+\ldots+\ell_{k+1}=j} \frac{\Gamma^{\ell_{k+1}}(1)}{\ell_{k+1}!}
\prod_{i=1}^k \roundbra{k-i+1}^{-(\ell_i+1)} }.
\end{equation}
\end{lemma}
\begin{proof}[Proof of Lemma~\ref{lem:H1s}]
As by definition $\mathcal H_1 = f^N$, we have
\begin{equation}
\lim_{s \rightarrow 1/2+k} \sqbra{\mathcal{H}_1(s)}^{(j)} = \lim_{s \rightarrow 1/2+k} \sqbra{f^{N} \roundbra{s}}^{(j)},
\end{equation}
with $f$ defined in Lemma~\ref{lemma:limfder}.
Using Leibniz' rule leads to
\begin{equation}
\lim_{s \rightarrow 1/2+k} \sqbra{\mathcal{H}_1(s)}^{(j)} = \sum_{j_1+j_2+\ldots+j_N=j} \binom{j}{j_1,j_2,\ldots,j_N} \prod_{1\leq t \leq N} \lim_{s \rightarrow 1/2+k} f^{(j_t)}(s).
\end{equation}
Applying Lemma~\ref{lemma:limfder} we obtain
\begin{equation}
\begin{aligned}
&\lim_{s \rightarrow 1/2+k} \sqbra{\mathcal{H}_1(s)}^{(j)}
\\
&  =  \sum_{j_1+j_2+\ldots+j_N=j} \binom{j}{j_1,j_2,\ldots,j_N} \\
&\quad \quad \cdot  \prod_{1\leq t \leq N} \sqbra{(-1)^{k+j_t-1} j_t! \roundbra{\sum_{\ell_1+\ell_2+\ldots+\ell_{k+1}=j_t} \frac{\Gamma^{\ell_{k+1}}(1)}{\ell_{k+1}!}
\prod_{i=1}^k \roundbra{k-i+1}^{-(\ell_i+1)}  }}
\\
&  =  \sum_{j_1+j_2+\ldots+j_N=j} \frac{j!}{j_1!j_2!\ldots j_N!} (-1)^{\sum_{t=1}^N (k+j_t-1)} j_1! j_2! \ldots j_N!
\\
&\quad \quad \cdot  \prod_{1\leq t \leq N}  \roundbra{\sum_{\ell_1+\ell_2+\ldots+\ell_{k+1}=j_t} \frac{\Gamma^{\ell_{k+1}}(1)}{\ell_{k+1}!}
\prod_{i=1}^k \roundbra{k-i+1}^{-(\ell_i+1)} }  \, ,
\end{aligned}
\end{equation}
which coincides with the lemma statement \eqref{eq:lem:H1s}.
\end{proof}

\section{Moment generating functions and Chernoff bounds}
There exists no moment generating function (MGF) for the product $X$ of $N$ iid standard Gaussian random variables.
However, clearly the moments exist and are given by
\begin{equation*}
\EE [X^k]
=
\EE\sqbra{\prod_{i=1}^N g_i^k}
=
\prod_{i=1}^N \EE[g_i^k]
=
\EE[g_1^k]^N
=
\begin{cases}
 \sqbra{(k-1)!!}^N
	& \text{ for even }k \, ,
\\
 0 & \text{ for odd }k \, ,
\end{cases}
\end{equation*}
with $n!!\coloneqq n(n-2)(n-4) \dots$ denoting the \emph{double factorial}.
Thus, the moments of $Y = X^2$ and $Z = |X|$ also exist.

The following proposition additionally provides the MGF on $[0,\infty)$ for the random variables $Y$ and $Z$. The proof of this proposition follows from properties~\eqref{int0xeG}, \eqref{zrhoG}, and \eqref{int0eG} in the Supplemental Materials and can be found in Ref.~\cite{SPIE_Gauss}.

\begin{proposition}[MGF of $Y$ and $Z$]\label{prop:MGF}
Let $\{g_i\}_{i=1}^N$ be a set of $N$ iid standard Gaussian random variables, i.e.\ $g_i \sim \mathcal{N}(0,1)$.
For the random variables $Y\coloneqq\prod_{i=1}^N g_i^2$ and $Z\coloneqq\prod_{i=1}^N |g_i|$ and all $t > 0$
\begin{align}
\EE\, \e^{-t Y} &  = \frac{1}{\pi^{N/2}}
    G^{N,1}_{1,N} \left(\frac{1}{2^N t} \Bigg|
        \begin{matrix}
          1 \\ 1/2,1/2, \dots, 1/2
        \end{matrix}
        \right) ,
\\
\EE\, \e^{-t Z} &  = \frac{1}{\pi^{\frac{N+1}{2}}}
    G^{N,2}_{2,N} \left( \frac{1}{2^{N-2} t^2} \Bigg|
        \begin{matrix}
           1, 1/2 \\ 1/2,1/2, \dots, 1/2
        \end{matrix}
        \right) .\label{eq:mgfZ}
\end{align}
\end{proposition}

\begin{remark}
All moments of $Y$ and $Z$ exist, and the MGF (given by the Meijer G-function) is smooth at the origin. Hence, we indeed have
\begin{align}
\EE[ Y^k ]
= (-1)^k \lim_{t \searrow 0}
\frac{\partial^k}{\partial t^k} \EE \, \e^{-t Y} \, .
\end{align}
\end{remark}
\begin{remark}
Knowing all the moments of the random variable $Y$, it seems trivial to compute the corresponding moment generating function for $t<0$
\begin{equation}
\EE \left[ e^{tY}\right] = \EE \left[ \sum_{j=0}^{\infty} \frac{t^j}{j!} Y^j\right].
\end{equation}
It is well known that one can exchange the expectation and the series if the series converges absolutely and this is not true in our case. 
Even more,
\begin{equation}
\lim_{j\to \infty} \EE \left[\frac{t^j}{j!} Y^j\right] = \infty
\end{equation}
for $N\geq 2$ and any $t>0$. 
Hence, the moment series does indeed not converge for any $t\neq 0$ and $N>1$. 
% \mk{better?}
\end{remark}

Even more, computing the MGF of $Y$ via the Meijer G-function allows for the following result which could be of independent interest.
\begin{corollary}
For $k \in \N_0$ and $N \in \N$ (with a convention that $(-1)!!=1$)
it holds that
\begin{equation}
\lim_{t \searrow 0}
\frac{1}{t^k}
    G^{N,1}_{1,N} \left(\frac{1}{2^N t} \Bigg|
        \begin{matrix}
          1-k \\ 1/2,1/2, \dots, 1/2
        \end{matrix}
        \right)
= \sqbra{\sqrt{\pi}(2k-1)!!}^N.
\end{equation}
\end{corollary}
\begin{proof}
The result follows from the following two observations
\begin{align}
\EE[ Y^k ] & = \EE\sqbra{\prod_{i=1}^N g_i^{2k}} = \prod_{i=1}^N \EE\sqbra{g_i^{2k}} = \sqbra{(2k-1)!!}^N 
\\ \nonumber
 \frac{\partial^k}{\partial t^k} \EE \, \e^{-t Y}  &= \frac{\partial^k}{\partial t^k} \frac{1}{\pi^{N/2}}
    G^{N,1}_{1,N} \left(\frac{1}{2^N t} \Bigg|
        \begin{matrix}
          1 \\ 1/2,1/2, \dots, 1/2
        \end{matrix}
        \right) \\
& = \frac{1}{\pi^{N/2}} \frac{(-1)^k}{t^k}
    G^{N,1}_{1,N} \left(\frac{1}{2^N t} \Bigg|
        \begin{matrix}
          1-k \\ 1/2,1/2, \dots, 1/2
        \end{matrix}
        \right),
\end{align}
where the last equality can be easily proven by induction.
\end{proof}

The analogous results then also hold for the random variable $Z$ but appear to be more technical.

The $k$-th moment of $Y$ and $Z$ can also be obtained as the $N$-th power of the $k$-th moment of a Gaussian random variable squared or of its absolute value, respectively.
However, often it is also important to know the MGF of the random variable, for instance, in order to obtain a tail bound for sums of iid copies of the random variable (Chernoff's inequality).

Next, using Proposition \ref{prop:MGF} we compute the Chernoff bound for random variables $\sum_{j=1}^M Y_j$ and $\sum_{j=1}^M Z_j$.

\begin{proposition}[Chernoff bound for $\sum_{j=1}^M Y_j$ and $\sum_{j=1}^M Z_j$]\label{Prop:ChernoffYZ}
  Let $\{g_{i,j}\}_{i=1,j=1}^{N,M}$ be a set of independent identically distributed standard Gaussian random variables (i.e., $g_{i,j} \sim \mathcal{N}(0, 1)$).
  Then, we have for $Y_j \coloneqq \prod_{i=1}^N g_{i,j}^2$ with $j \in [M]$
  \begin{align}
    \PP\left(\sum_{j=1}^M Y_j \leq t \right) \le \frac{1}{\pi^{\frac{MN}{2}}} \min \left\{ F_Y(t), G_Y(t) \right\},
  \end{align}
  where
  \begin{align}
    F_Y(t) &\coloneqq  e^{\frac{t}{2^N}}  \sqbra{G^{N,1}_{1,N} \left( 1 \Bigg| \begin{matrix} 1 \\ 1/2,1/2, \dots, 1/2 \end{matrix} \right)}^M , \label{Eq:Yapprox1}\\
    G_Y(t) &\coloneqq  e^{\frac{M}{2}}  \sqbra{G^{N,1}_{1,N} \left(\frac{t}{2^{N-1}M} \Bigg| \begin{matrix} 1 \\ 1/2,1/2, \dots, 1/2 \end{matrix} \right)}^M \label{Eq:Yapprox2}.
  \end{align}

  Furthermore, for $Z_j \coloneqq \prod_{i=1}^N \abs{g_{i,j}}$ with $j \in \sqbra{M}$ we have
  \begin{align}
    \PP\left(\sum_{j=1}^M Z_j \leq t \right) \le \frac{1}{\pi^{\frac{M(N+1)}{2}}} \min & \left\{ F_Z(t), G_Z(t) \right\},
  \end{align}
  where
  \begin{align}
    F_Z(t) \coloneqq  e^{t \cdot 2^{-\frac{N-2}{2}}}  \sqbra{G^{N,2}_{2,N} \left( 1 \Bigg| \begin{matrix} 1,1/2 \\ 1/2,1/2, \dots, 1/2 \end{matrix} \right)}^M ,\label{Eq:Zapprox1} \\
    G_Z(t) \coloneqq  e^{M}  \sqbra{G^{N,2}_{2,N} \left(\frac{t^2}{2^{N-2}M^2} \Bigg| \begin{matrix} 1,1/2 \\ 1/2,1/2, \dots, 1/2 \end{matrix} \right)}^M. \label{Eq:Zapprox2}
  \end{align}
  \end{proposition}

%%%%%%%%%%%%%%%%%%%%%%%%
% \newpage
% $\ $ \vspace{1cm} $\ $
\afterpage{%
\newgeometry{left = 2cm, right = 2cm, top = 4cm}
\begin{figure}[hp!]
  \subfloat[$N = 4$]{
     \includegraphics[width=0.32\textwidth]{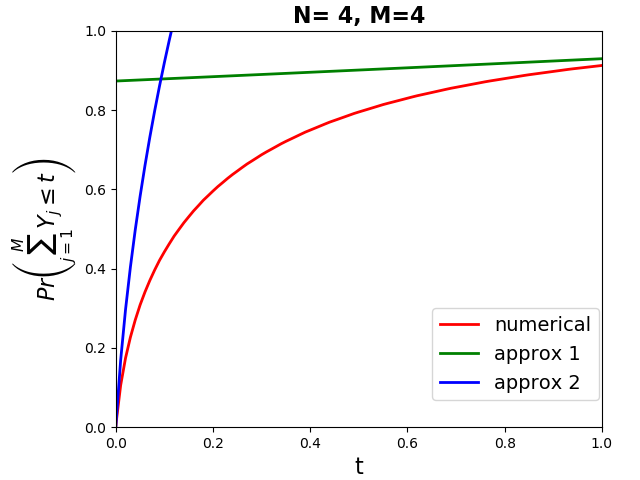}
  \hfill   \includegraphics[width=0.32\textwidth]{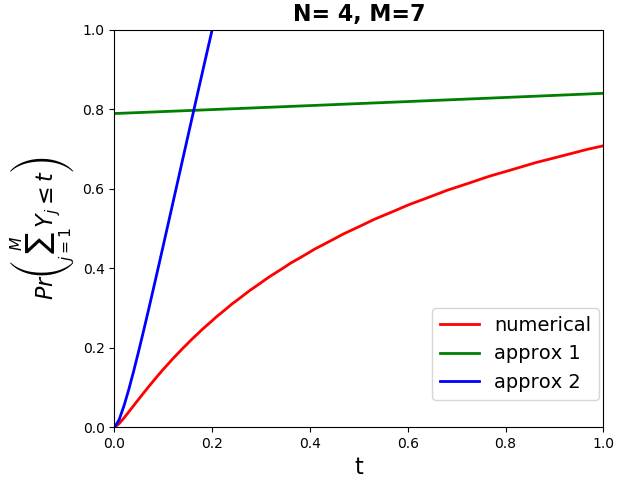}
      \hfill
  \includegraphics[width=0.32\textwidth]{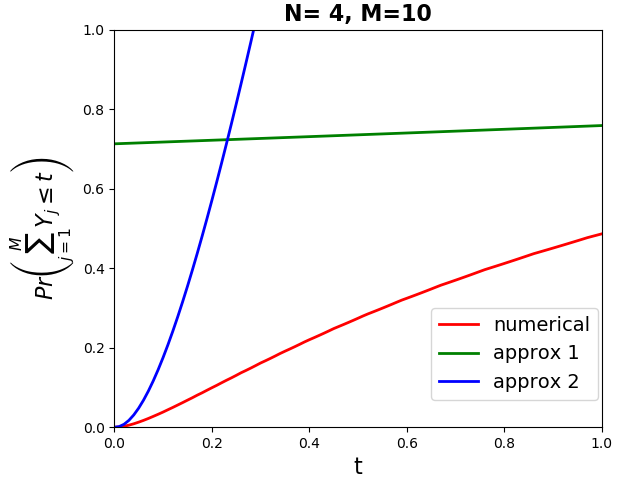}
       } \\
   \subfloat[$N = 5$]{ %
     \includegraphics[width=0.32\textwidth]{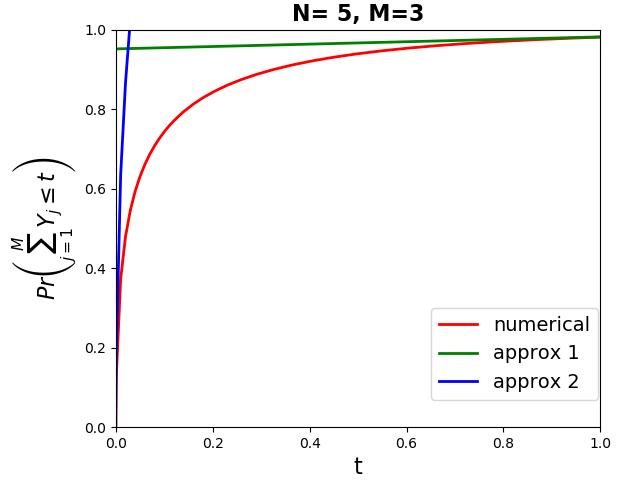}
 \hfill
         \includegraphics[width=0.32\textwidth]{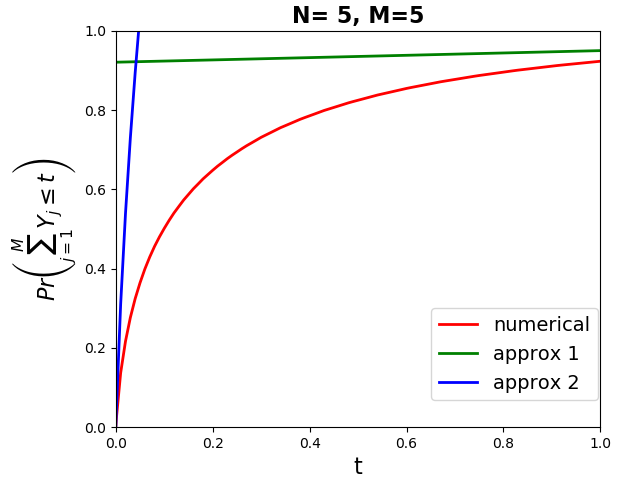}
   \hfill
  \includegraphics[width=0.32\textwidth]{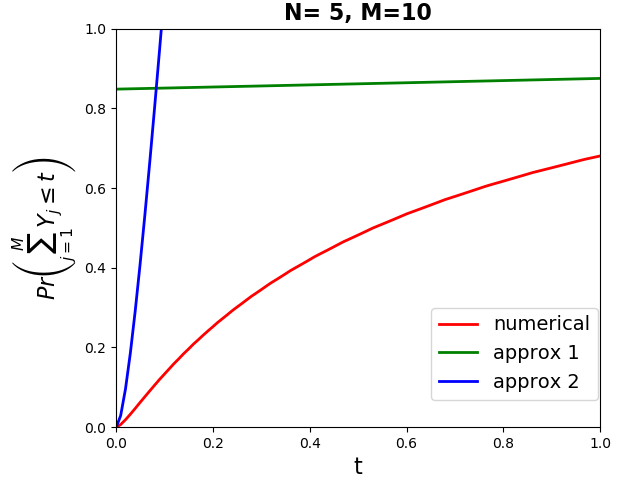}
       }\\
   \subfloat[$N = 7$] {
  \includegraphics[width=0.32\textwidth]{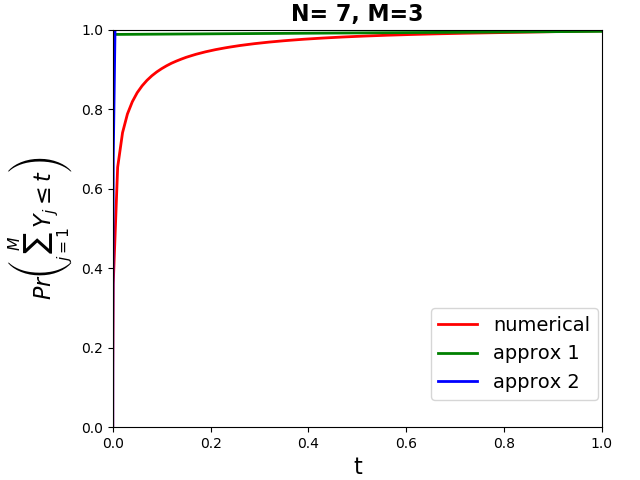}
   \hfill
  \includegraphics[width=0.32\textwidth]{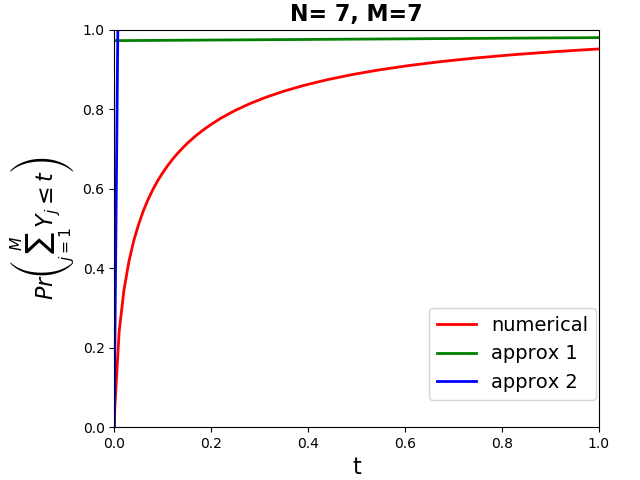}
   \hfill
  \includegraphics[width=0.32\textwidth]{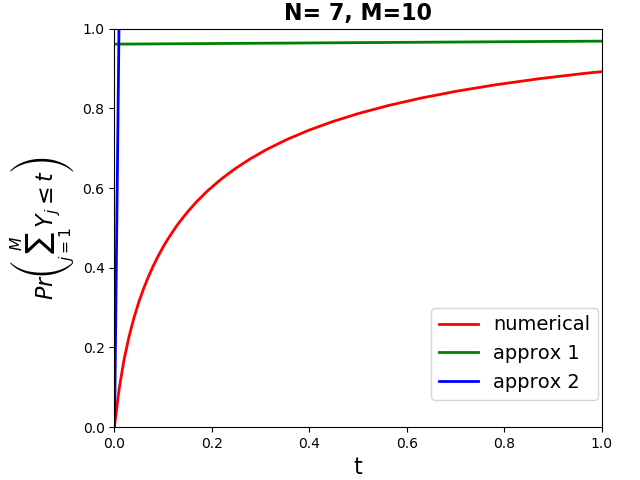}
       }\\[1em]
  \caption{%
    Numerical comparison of the Chernoff bounds stated in Proposition~\ref{Prop:ChernoffYZ}.
    The green and blue lines show the bounds in %Eqs.~
    \eqref{Eq:Yapprox1} and \eqref{Eq:Yapprox2}, respectively.
    The red line shows~\eqref{Eq:ChernoffMsumY} where for each $t$ the minimum over $\theta$ is approximated numerically by a minimum over a finite number of points.
    Here, we chose the values of $\theta$, over which the minimum is taken, evenly spaced on a logarithmic scale between $2^{-{N/2+2}}$ and $M/t$.
    \label{Fig:ComparisonChernoffY}
  }
 \end{figure}

\newpage
% $\ $ \vspace{1cm} $\ $
%%%%%%%%%%%%%%%%%%%%%%%%
\begin{figure}[hp!]
    \subfloat[$N = 4$]{
       \includegraphics[width=0.32\textwidth]{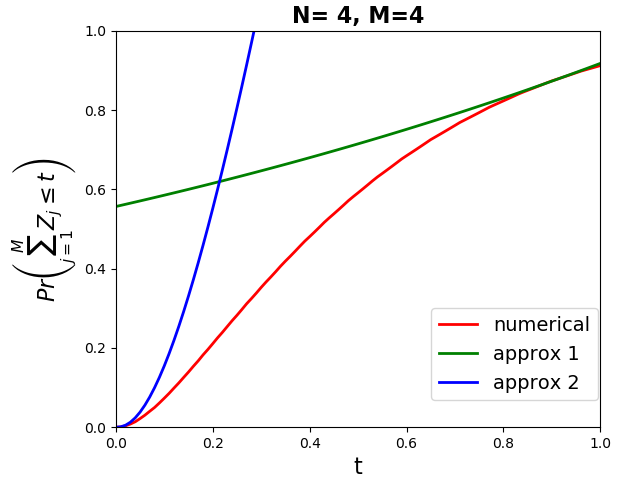}
    \hfill   \includegraphics[width=0.32\textwidth]{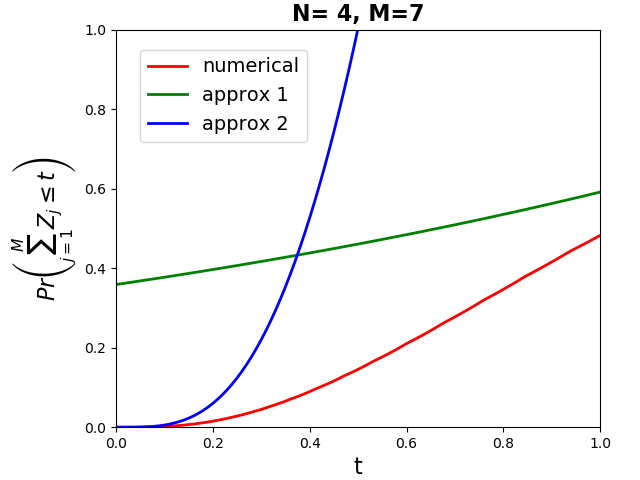}
        \hfill
    \includegraphics[width=0.32\textwidth]{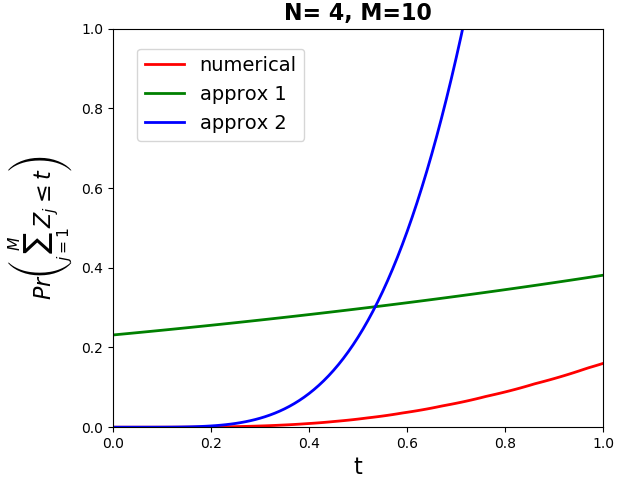}
         } \\
     \subfloat[$N = 5$]{ %
       \includegraphics[width=0.32\textwidth]{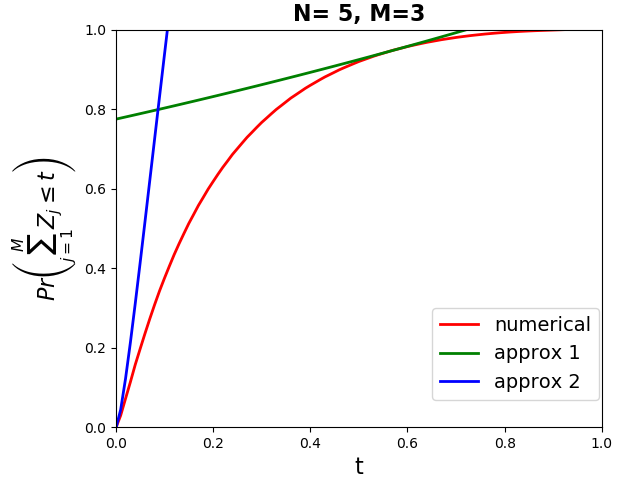}
   \hfill
           \includegraphics[width=0.32\textwidth]{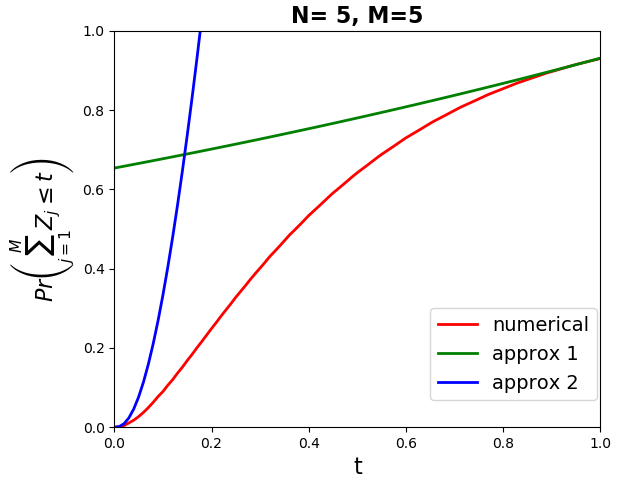}
     \hfill
    \includegraphics[width=0.32\textwidth]{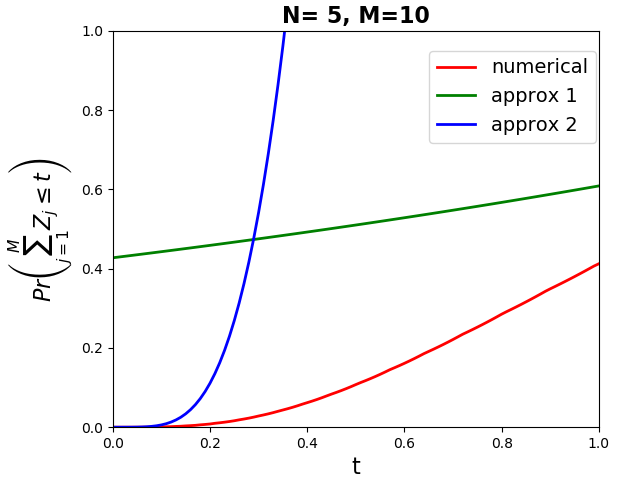}
         }\\
     \subfloat[$N = 7$] {
    \includegraphics[width=0.32\textwidth]{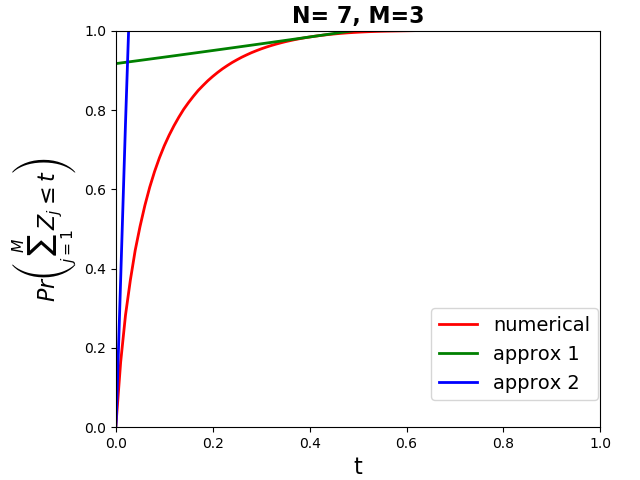}
     \hfill
    \includegraphics[width=0.32\textwidth]{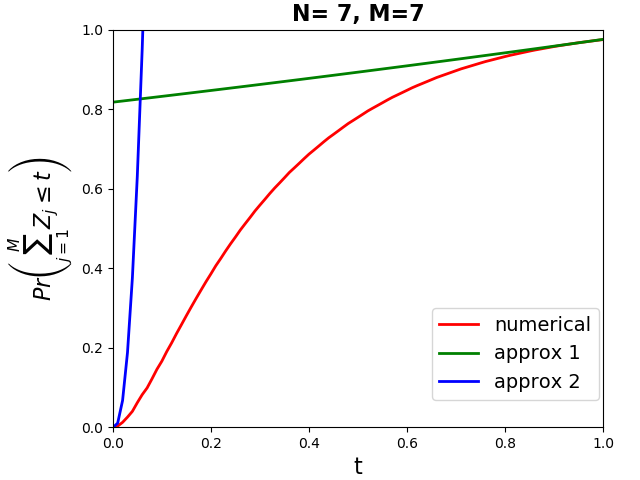}
     \hfill
    \includegraphics[width=0.32\textwidth]{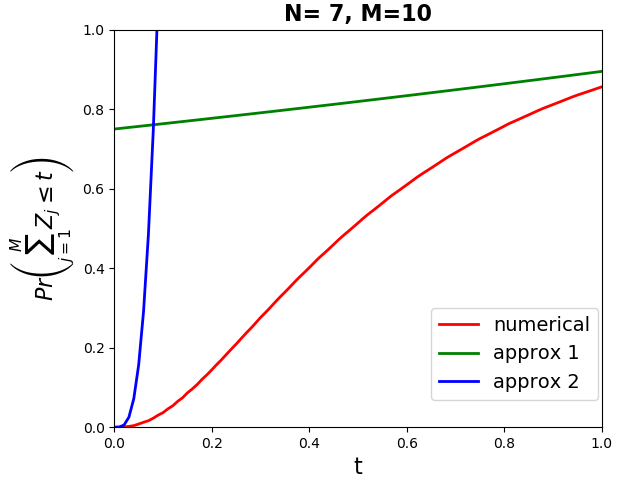}
         }\\[1em]
   \caption{%
     The analogue of Fig.~\ref{Fig:ComparisonChernoffY} for the random variable $\sum_{j=1}^M Z_j$.
     The green and blue lines show the bounds in %Eqs.~
     \eqref{Eq:Zapprox1} and \eqref{Eq:Zapprox2}, respectively.
     \label{Fig:ComparisonChernoffZ}
  }
\end{figure}
\clearpage
\restoregeometry
} % end of \afterpage{...} material

\begin{proof}[Proof of Proposition \ref{Prop:ChernoffYZ}]
For $\theta>0$, by Chernoff's inequality, we have
\begin{align}
\PP \left( \sum_{j=1}^M Y_j \leq t \right) & \leq  \min_{\theta>0}\,  \e^{\theta t} \, \prod_{j=1}^M \EE \left( \e^{-\theta  Y_j} \right) =  \min_{\theta >0}\,  \e^{\theta t} \,\sqbra{\EE \left( \e^{-\theta  Y_1} \right) }^M .
\end{align}
Applying Proposition~\ref{prop:MGF} we obtain
\begin{equation}
\begin{aligned}
\PP \left( \sum_{j=1}^M Y_j \leq t \right) & \leq  \min_{\theta >0}\,  \e^{\theta  t} \,\sqbra{ \frac{1}{\pi^{N/2}}
    G^{N,1}_{1,N} \left(\frac{1}{2^N \theta } \Bigg|
        \begin{matrix}
          1 \\ 1/2,1/2, \dots, 1/2
        \end{matrix}
        \right) }^M \\
& = \frac{1}{\pi^{{\frac{MN}{2}}}} \min_{\theta >0}\,  \e^{\theta  t} \,\sqbra{
    G^{N,1}_{1,N} \left(\frac{1}{2^N \theta } \Bigg|
        \begin{matrix}
          1 \\ 1/2,1/2, \dots, 1/2
        \end{matrix}
        \right) }^M . \label{Eq:ChernoffMsumY}
\end{aligned}
\end{equation}
Next, we define a function
\begin{equation}
f_Y(\theta ) \coloneqq \e^{\theta t} \,\sqbra{
    G^{N,1}_{1,N} \left(\frac{1}{2^N \theta } \Bigg|
        \begin{matrix}
          1 \\ 1/2,1/2, \dots, 1/2
        \end{matrix}
        \right) }^M \, .
\end{equation}
To compute the minimizer, we calculate the corresponding derivative,
\begin{align*}
& \frac{d}{d \theta } f_Y(\theta )  \\
& = t \cdot \e^{\theta t} \,\sqbra{
    G^{N,1}_{1,N} \left(\frac{1}{2^N \theta} \Bigg|
        \begin{matrix}
          1 \\ 1/2,1/2, \dots, 1/2
        \end{matrix}
        \right) }^M  \\
        &
         + \e^{\theta t} \, M \cdot \sqbra{
    G^{N,1}_{1,N} \left(\frac{1}{2^N \theta } \Bigg|
        \begin{matrix}
          1 \\ 1/2,1/2, \dots, 1/2
        \end{matrix}
        \right) }^{M-1} \frac{d}{d \theta }     G^{N,1}_{1,N} \left(\frac{1}{2^N \theta } \Bigg|
        \begin{matrix}
          1 \\ 1/2,1/2, \dots, 1/2
        \end{matrix}
        \right) \\
& = \e^{\theta t}\, \sqbra{
    G^{N,1}_{1,N} \left(\frac{1}{2^N \theta } \Bigg|
        \begin{matrix}
          1 \\ 1/2,1/2, \dots, 1/2
        \end{matrix}
        \right) }^{M-1} 
        \\
        &\phantom{=.}  \cdot \curlybra{t \cdot
    G^{N,1}_{1,N} \left(\frac{1}{2^N \theta } \Bigg|
        \begin{matrix}
          1 \\ 1/2,1/2, \dots, 1/2
        \end{matrix}
        \right) + M \frac{d}{d \theta}    G^{N,1}_{1,N} \left(\frac{1}{2^N \theta } \Bigg|
        \begin{matrix}
          1 \\ 1/2,1/2, \dots, 1/2
        \end{matrix}
        \right)}  \, .
\end{align*}
Since the moment generating function, when it exists, is positive
(i.e.\ $ \EE [\e^{\theta Y}] \geq  \e^{\mu \theta }$ with $\mu = \EE(Y)$)
we have $\frac{d}{d\theta }f_Y(\theta )=0$ if and only if
\begin{equation}
 t \cdot
    G^{N,1}_{1,N} \left(\frac{1}{2^N \theta } \Bigg|
        \begin{matrix}
          1 \\ 1/2,1/2, \dots, 1/2
        \end{matrix}
        \right) + M \frac{d}{d\theta }    G^{N,1}_{1,N} \left(\frac{1}{2^N \theta} \Bigg|
        \begin{matrix}
          1 \\ 1/2,1/2, \dots, 1/2
        \end{matrix}
        \right) = 0 \, .
\end{equation}
Applying \eqref{MGder} from the Supplemental Materials leads to
\begin{align}
\frac{d}{d\theta }    G^{N,1}_{1,N} \left(\frac{1}{2^N \theta } \Bigg|
        \begin{matrix}
          1 \\ 1/2, \dots, 1/2
        \end{matrix}
        \right) 
&= 
        \roundbra{\frac{1}{2^N \theta }}^{-1} \cdot \frac{-1}{2^N \theta ^2}
 G^{N,1}_{1,N} \left(\frac{1}{2^N \theta} \Bigg|
        \begin{matrix}
         0 \\ 1/2,1/2, \dots, 1/2
        \end{matrix}
        \right)  
\nonumber \\
&= 
\frac{-1}{\theta}
G^{N,1}_{1,N} \left(\frac{1}{2^N \theta} \Bigg|
        \begin{matrix}
         0 \\ 1/2,1/2, \dots, 1/2
        \end{matrix}
        \right).
\end{align}
Therefore, we need to find $\theta$ such that
\begin{align}\label{eq:Chernoff_MXeq}
 & t \cdot
    G^{N,1}_{1,N} \left(\frac{1}{2^N \theta } \Bigg|
        \begin{matrix}
          1 \\ 1/2,1/2, \dots, 1/2
        \end{matrix}
        \right) - \frac{M}{\theta}  \, G^{N,1}_{1,N} \left(\frac{1}{2^N \theta} \Bigg|
        \begin{matrix}
          0 \\ 1/2,1/2, \dots, 1/2
        \end{matrix}
        \right) = 0 \, 
\\\nonumber
 \Leftrightarrow  \,
 & t \cdot
    G^{1,N}_{N,1} \left(2^N \theta \Bigg|
        \begin{matrix}
           1/2,1/2, \dots, 1/2 \\ 0
        \end{matrix}
        \right) - \frac{M}{\theta}  \, G^{1,N}_{N,1} \left(2^N \theta \Bigg|
        \begin{matrix}
          1/2,1/2, \dots, 1/2 \\ 1
        \end{matrix}
        \right) = 0 \, ,
\end{align}
where the second equation follows from %Eq.~
\eqref{reverseG}.
However, directly solving %Eq.~
\eqref{eq:Chernoff_MXeq} for $\theta$ is still intractable.
Our idea is to approximate both Meijer G-functions appearing in said equation with the lowest order of a power-log series expansion.
In this way we obtain a good-enough but not necessarily optimal choice of $\theta$ for the bound on the CDF of $\sum_{j=1}^M Y_j$ from %Eq.~
\eqref{Eq:ChernoffMsumY}.
The necessary power-log series expansion is provided in the next Lemma.
The proof can be found in Appendix~\ref{App:Proofs}.
\begin{lemma}[Power-log series expansion related to $\sum_{j=1}^M Y_j$]\label{lemma:sumYpowerlog}
For $x \in \{0,1\}$
\begin{equation}
 G^{1,N}_{N,1} \left(z \Bigg|
        \begin{matrix}
           1/2,1/2, \dots, 1/2 \\ x
        \end{matrix}
        \right)
        =
 \sum_{k=0}^{\infty} z^{-1/2-k} \sum_{j=0}^{N-1} H_{kj}^x \cdot \sqbra{\log (z)}^j ,
\end{equation}
where
\begin{equation}
\begin{aligned}
H_{kj}^x &= \frac{(-1)^{Nk-1}}{j!} \sum_{n=j}^{N-1} \frac{(-1)^{n-j}}{(n-j)!} \cdot \Gamma^{(n-j)} (x+1/2+k) \\
&\cdot \sum_{j_1+\ldots+j_N=N-1-n}\prod_{t=1}^N \curlybra{\sum_{\ell_1+\ldots+\ell_{k+1}=j_t} \frac{\Gamma^{(\ell_{k+1})}(1)}{\ell_{k+1}!} \prod_{i=1}^{k-1} (k-i+1)^{-(\ell_i+1)}}
\end{aligned}
\end{equation}
\end{lemma}

Using the full expansion it is still difficult to obtain the solution of \eqref{eq:Chernoff_MXeq}. Therefore, we take the simplest approximation ($k=0$, $j=N-1$), i.e.\
\begin{equation}
 G^{1,N}_{N,1} \left(2^N \theta \Bigg|
        \begin{matrix}
           1/2,1/2, \dots, 1/2 \\ x
        \end{matrix}
        \right)
        \approx \roundbra{2^N \theta}^{-1/2} H_{0,N-1}^x \cdot \sqbra{\log (2^N \theta)}^{N-1} ,
\end{equation}
where
\begin{align}
H_{0,N-1}^0 & =\frac{-1}{(N-1)!} \Gamma(1/2) \\
H_{0,N-1}^1 & =\frac{-1}{(N-1)!} \Gamma(1+1/2) = \frac{-1}{(N-1)!} \frac{1}{2}\Gamma(1/2) = \frac{1}{2} H_{0,N-1}^0
\end{align}
and compute the corresponding solution. The approximation becomes better as the argument grows.
Thus, since $z=2^N \theta$, the approximation improves as $N$ grows.
Next, we solve the following equation for $\theta$,
\begin{equation}
 t \cdot \roundbra{2^N \theta}^{-1/2} H_{0,N-1}^0 \cdot \sqbra{\log (2^N \theta)}^{N-1} - \frac{M}{\theta} \roundbra{2^N \theta}^{-1/2}  \frac{1}{2}H_{0,N-1}^0 \cdot \sqbra{\log (2^N \theta)}^{N-1} = 0
\end{equation}
% \Leftrightarrow \,
which is equivalent to solving
\begin{equation}
 \roundbra{2^N \theta}^{-1/2} H_{0,N-1}^0 \cdot \sqbra{\log (2^N \theta)}^{N-1} \roundbra{t  - \frac{M}{2\theta }} = 0.
\end{equation}
The solutions are $\theta=\frac{1}{2^N}$ and $\theta=\frac{M}{2t}$.

Plugging this result into \eqref{Eq:ChernoffMsumY} we obtain
\begin{align}
\PP \roundbra{\sum_{j=1}^M Y_j \leq t} \leq \frac{1}{\pi^{\frac{MN}{2}}} \min & \left\{e^{\frac{t}{2^N}}   \sqbra{G^{N,1}_{1,N} \left( 1 \Bigg|
        \begin{matrix}
          1 \\ 1/2,1/2, \dots, 1/2
        \end{matrix}
        \right)}^M , \right. 
        \nonumber\\
        & \quad \quad \left.
 e^{\frac{M}{2}}  \sqbra{G^{N,1}_{1,N} \left(\frac{t}{2^{N-1}M} \Bigg|
        \begin{matrix}
          1 \\ 1/2,1/2, \dots, 1/2
        \end{matrix}
        \right)}^M
 \right\} \, .
\end{align}

%%%%%%%%%%%%%%%%%%%%%%

Similarly, we derive the Chernoff bound for the random variable $\sum_{j=1}^M Z_j$.
With Proposition~\ref{prop:MGF} we obtain
\begin{align} %\label{Eq:ChernoffMsumZ}
\PP \left( \sum_{j=1}^M Z_j \leq t \right) & \leq \min_{\theta>0}\,  \e^{\theta t} \,\sqbra{\EE  \e^{-\theta Z_1}  }^M 
\nonumber \\ \nonumber
 & \leq  \min_{\theta >0}\,  \e^{\theta t} \,\sqbra{ \frac{1}{\pi^{\frac{N+1}{2}}}
    G^{N,2}_{2,N} \left(\frac{1}{2^{N-2} \theta^2} \Bigg|
        \begin{matrix}
          1, 1/2 \\ 1/2,1/2, \dots, 1/2
        \end{matrix}
        \right) }^M \\
& = \frac{1}{\pi^{{\frac{M(N+1)}{2}}}} \min_{\theta>0}\,  \e^{\theta t} \,\sqbra{
    G^{N,2}_{2,N} \left(\frac{1}{2^{N-2} \theta^2} \Bigg|
        \begin{matrix}
         1, 1/2 \\ 1/2,1/2, \dots, 1/2
        \end{matrix}
        \right) }^M .  \label{Eq:ChernoffMsumZ}
\end{align}
Next, we define a function
\begin{equation}
f_Z(\theta) \coloneqq \e^{\theta t} \,\sqbra{
    G^{N,2}_{2,N} \left(\frac{1}{2^{N-2} \theta^2} \Bigg|
        \begin{matrix}
          1,1/2 \\ 1/2,1/2, \dots, 1/2
        \end{matrix}
        \right) }^M \, .
\end{equation}
To compute the minimizer, we calculate the corresponding derivative,
\begin{equation}
\begin{aligned}
&\phantom{=.}
\frac{d}{d \theta} f_Z(\theta)
\\
&= t \cdot \e^{\theta t} \,\sqbra{
    G^{N,2}_{2,N} \left(\frac{1}{2^{N-2} \theta^2} \Bigg|
        \begin{matrix}
          1,1/2 \\ 1/2,1/2, \dots, 1/2
        \end{matrix}
        \right) }^M  
\\ 
&\phantom{=.}
                          + M \e^{\theta t} \cdot  \sqbra{
    G^{N,2}_{2,N} \left(\frac{1}{2^{N-2} \theta^2} \Bigg|
        \begin{matrix}
          1,1/2 \\ 1/2, \dots, 1/2
        \end{matrix}
        \right) }^{M-1}
        %\\
        %& \cdot %\quad \quad \quad \cdot
        \frac{d}{d \theta}     G^{N,2}_{2,N} \left(\frac{1}{2^{N-2} \theta^2} \Bigg|
        \begin{matrix}
          1,1/2 \\ 1/2, \dots, 1/2
        \end{matrix}
        \right) \\
& = \e^{\theta t}\, \sqbra{
    G^{N,2}_{2,N} \left(\frac{1}{2^{N-2} \theta^2} \Bigg|
        \begin{matrix}
         1, 1/2 \\ 1/2,1/2, \dots, 1/2
        \end{matrix}
        \right) }^{M-1} 
\\
&\phantom{=.}  
\cdot \curlybra{t \cdot
    G^{N,2}_{2,N} \left(\frac{1}{2^{N-2} \theta^2} \Bigg|
        \begin{matrix}
         1, 1/2 \\ 1/2, \dots, 1/2
        \end{matrix}
        \right) + M \frac{d}{d \theta}    G^{N,2}_{2,N} \left(\frac{1}{2^{N-2} \theta^2} \Bigg|
        \begin{matrix}
          1, 1/2 \\ 1/2, \dots, 1/2
        \end{matrix}
        \right)}  \, .     
\end{aligned}
\end{equation}
Similarly as before, we apply \eqref{MGder} to obtain
\begin{equation}
\begin{aligned}
&\phantom{=.}
\frac{d}{d \theta}    G^{N,2}_{2,N} \left(\frac{1}{2^{N-2} \theta^2} \Bigg|
        \begin{matrix}
          1/2,1 \\ 1/2,1/2, \dots, 1/2
        \end{matrix}
        \right)
        \\
& =  \roundbra{\frac{1}{2^{N-2} \theta^2}}^{-1} \cdot \frac{-2}{2^{N-2} \theta^3}
% \\
% &\quad
\cdot
   G^{N,2}_{2,N} \left(\frac{1}{2^{N-2} \theta^2} \Bigg|
        \begin{matrix}
         0, 1/2 \\ 1/2,1/2, \dots, 1/2
        \end{matrix}
        \right)      \\
& =  \frac{-2}{\theta} \cdot
      G^{N,2}_{2,N} \left(\frac{1}{2^{N-2} \theta^2} \Bigg|
        \begin{matrix}
         0, 1/2 \\ 1/2,1/2, \dots, 1/2
        \end{matrix}
        \right) \,  .
\end{aligned}
\end{equation}
Thus, we need to find $\theta>0$ such that
\begin{equation}
 t \cdot
    G^{N,2}_{2,N} \left(\frac{1}{2^{N-2} \theta^2} \Bigg|
        \begin{matrix}
          1,1/2 \\ 1/2,1/2, \dots, 1/2
        \end{matrix}
        \right) - \frac{2M}{\theta}
        G^{N,2}_{2,N} \left(\frac{1}{2^{N-2} \theta^2} \Bigg|
        \begin{matrix}
         0, 1/2 \\ 1/2,1/2, \dots, 1/2
        \end{matrix}
        \right) = 0 
\end{equation}
which is equivalent to
\begin{equation}
 t \cdot
    G^{2,N}_{N,2} \left(2^{N-2} \theta^2 \Bigg|
        \begin{matrix}
           1/2,1/2, \dots, 1/2 \\ 0, 1/2
        \end{matrix}
        \right) - \frac{2M}{\theta}    G^{2,N}_{N,2} \left(2^{N-2} \theta^2 \Bigg|
        \begin{matrix}
         1/2,1/2, \dots, 1/2 \\ 1/2,1
        \end{matrix}
        \right) = 0 \, ,      \label{eq:Chernoff_MZ}
\end{equation}
where the second equation follows from %Eq.~
\eqref{reverseG}.
As already experienced in the analysis of the random variable $\sum_{j=1}^M Y_j$, an analytic solution for the optimal value of $\theta$ from %Eq.~
\eqref{eq:Chernoff_MZ} is infeasible.
Once again, we solve the approximate equality obtained by replacing the Meijer G-functions with their lowest order in the power-log series expansion.
In this way we obtain a good-enough but not necessarily optimal choice of $\theta$ for the bound on the CDF of $\sum_{j=1}^M Z_j$ from %Eq.~
\eqref{Eq:ChernoffMsumZ}.
We postpone the proof of the following Lemma, which contains said power-log series expansion, to Sec.~\ref{App:Proofs}.
\begin{lemma}[Power-log series expansion related to $\sum_{j=1}^M Z_j$]\label{lemma:sumZpowerlog}
For $x \in \{0,1\}$
\begin{equation}
 G^{2,N}_{N,2} \left(z \Bigg|
        \begin{matrix}
           1/2,1/2, \dots, 1/2 \\ 1/2, x
        \end{matrix}
        \right)
        =
 \sum_{k=0}^{\infty} z^{-1/2-k} \sum_{j=0}^{N-1} \bar{H}_{kj}^x \cdot \sqbra{\log (z)}^j ,
\end{equation}
where
\begin{equation}
\begin{aligned}
\bar{H}_{kj}^x 
&=
\frac{(-1)^{Nk-1}}{j!} \sum_{n=j}^{N-1} \frac{(-1)^{n-j}}{(n-j)!} \cdot \roundbra{\sum_{i=0}^{n-j} \binom{n-j}{i} \Gamma^{(n-j-i)}(1+k) \Gamma^{(i)} (x+1/2+k)} 
\\
& \phantom{=.}
\sum_{j_1+\ldots+j_N=N-1-n}\prod_{t=1}^N \curlybra{\sum_{\ell_1+\ldots+\ell_{k+1}=j_t} \frac{\Gamma^{(\ell_{k+1})}(1)}{\ell_{k+1}!} \prod_{i=1}^{k-1} (k-i+1)^{-(\ell_i+1)}}\, .
\end{aligned}
\end{equation}
\end{lemma}

Using the full expansion it is still difficult to obtain the solution of \eqref{eq:Chernoff_MZ}. Therefore, we take the simplest approximation ($k=0$, $j=N-1$), i.e.\
\begin{equation}
 G^{2,N}_{N,2} \left(2^{N-2} \theta^2 \Bigg|
        \begin{matrix}
           1/2,1/2, \dots, 1/2 \\ 1/2, x
        \end{matrix}
        \right)
        \approx \roundbra{2^{N-2} \theta^2}^{-1/2} \bar{H}_{0,N-1}^x \cdot \sqbra{\log (2^{N-2} \theta^2)}^{N-1} ,
\end{equation}
where
\begin{align}
\bar{H}_{0,N-1}^0 & =\frac{-1}{(N-1)!} \Gamma(1)\Gamma(1/2)   \\
\bar{H}_{0,N-1}^1 & =\frac{-1}{(N-1)!} \Gamma(1) \Gamma(1+1/2) = \frac{-1}{(N-1)!} \Gamma(1) \frac{1}{2}\Gamma(1/2) = \frac{1}{2} \bar{H}_{0,N-1}^0
\end{align}
and compute the corresponding solution. The approximation of the aforementioned Meijer G-function is better as the argument grows.
As $z=2^{N-2} \theta^2$, the approximation improves with growing $N$.
Thus, plugging these approximations in \eqref{eq:Chernoff_MZ} we need to solve the following equation for $\theta$,
\begin{align}
& t \cdot \roundbra{2^{N-2} \theta^2}^{-1/2} \bar{H}_{0,N-1}^0 \cdot \sqbra{\log (2^{N-2} \theta^2)}^{N-1}  
\nonumber \\ \nonumber 
& \quad \quad - \frac{2M}{\theta} \roundbra{2^{N-2} \theta^2}^{-1/2}  \frac{1}{2}\bar{H}_{0,N-1}^0 \cdot \sqbra{\log (2^{N-2} \theta^2)}^{N-1} = 0 \\
 \Leftrightarrow \,
& \roundbra{2^{N-2} \theta^2}^{-1/2} \bar{H}_{0,N-1}^0 \cdot \sqbra{\log (2^{N-2} \theta^2)}^{N-1} \roundbra{t - \frac{M}{\theta}} = 0.
\end{align}
The solutions are $\theta=\frac{1}{2^{\frac{N-2}{2}}}=2^{-\frac{N-2}{2}}$ and $\theta=\frac{M}{t}$.

Plugging the solutions into \eqref{Eq:ChernoffMsumZ} we obtain
\begin{align}
\PP \roundbra{\sum_{j=1}^M Z_j \leq t} \leq \frac{1}{\pi^{\frac{M(N+1)}{2}}} \min & \left\{e^{t \cdot 2^{-\frac{N-2}{2}}}  \sqbra{G^{N,2}_{2,N} \left( 1 \Bigg|
        \begin{matrix}
          1,1/2 \\ 1/2,1/2, \dots, 1/2
        \end{matrix}
        \right)}^M ,\right. 
\nonumber\\
        & \left.
 e^{M}  \sqbra{G^{N,2}_{2,N} \left(\frac{t^2}{2^{N-2}M^2} \Bigg|
        \begin{matrix}
          1,1/2 \\ 1/2,1/2, \dots, 1/2
        \end{matrix}
        \right)}^M
 \right\}.
\end{align}

\end{proof}

\begin{remark}
In Figure~\ref{Fig:ComparisonChernoffY} and Figure~\ref{Fig:ComparisonChernoffZ} we compare the Chernoff bounds for the random variable $Y=\sum_{j=1}^M Y_j$ and $Z=\sum_{j=1}^M Z_j$, respectively. In particular, we compare the numerical minimum in \eqref{Eq:ChernoffMsumY} and \eqref{Eq:ChernoffMsumZ}  with the bounds obtained after truncating the power log series of the Meijer~G-functions obtained in Lemma~\ref{lemma:sumYpowerlog} and Lemma~\ref{lemma:sumZpowerlog}, respectively.
\end{remark}

%%%%%%%%%%%%%%%%%%%%%%
\section{Conclusion and outlook}
We have considered the three random variables $X$, $Y$, and $Z$ given by the products of $N$ Gaussian iid random variables, their squares, and their absolute values.
First, we have expressed their CDFs in terms of Meijer G-functions and provided the corresponding power-log series expansions.
Numerically, we demonstrated that a truncation of these series at the lowest orders yields quite tight approximations.
Second, we calculated the MGFs of $Y$ and $Z$ also in terms of Meijer G-functions.
As a consequence, all moments of $Y$ and $Z$ can be expressed in terms of these functions, which yields a new identity for certain Meijer-G functions.
We also provided the corresponding Chernoff bounds for sums of iid copies of $Y$ and $Z$.

Providing explicit error bounds for the truncated power-log series and tight upper and lower bounds to the CDFs of $X$, $Y$, and $Z$ is left for future research.
The main difficulty in this endeavor seems to be the following variant of a ``sign problem'':
The summands of the expansions of the Meijer G-functions are relatively large, have fluctuating signs and cancel out to give a small value in the end.

%%%%%%%%%%%%%%%%%%%%%%
\section{Acknowledgments}
We would like to thank David Gross for advice, Claudio Cacciapuoti for fruitful discussions on special functions, and Peter Jung for discussions on connections to compressed sensing.

The work of ZS and DS has been supported by the Excellence Initiative of the German Federal and State Governments (Grant 81), the ARO under contract W911NF-14-1-0098 (Quantum Characterization, Verification, and Validation), and the DFG projects GRO 4334/1,2 (SPP1798 CoSIP).
The work of MK was funded by the National Science Centre, Poland within the project Polonez (2015/19/P/ST2/03001) which has received funding from the European Union’s Horizon 2020 research and innovation programme under the Marie Sk{\l}odowska-Curie grant agreement No 665778.

%%%%%%%%%%%%%%%%%%%%%%
\bibliographystyle{apsrev4-1}
\bibliography{report}

%merlin.mbs apsrev4-1.bst 2010-07-25 4.21a (PWD, AO, DPC) hacked
%Control: key (0)
%Control: author (72) initials jnrlst
%Control: editor formatted (1) identically to author
%Control: production of article title (-1) disabled
%Control: page (0) single
%Control: year (1) truncated
%Control: production of eprint (0) enabled
\begin{thebibliography}{16}%
\makeatletter
\providecommand \@ifxundefined [1]{%
 \@ifx{#1\undefined}
}%
\providecommand \@ifnum [1]{%
 \ifnum #1\expandafter \@firstoftwo
 \else \expandafter \@secondoftwo
 \fi
}%
\providecommand \@ifx [1]{%
 \ifx #1\expandafter \@firstoftwo
 \else \expandafter \@secondoftwo
 \fi
}%
\providecommand \natexlab [1]{#1}%
\providecommand \enquote  [1]{``#1''}%
\providecommand \bibnamefont  [1]{#1}%
\providecommand \bibfnamefont [1]{#1}%
\providecommand \citenamefont [1]{#1}%
\providecommand \href@noop [0]{\@secondoftwo}%
\providecommand \href [0]{\begingroup \@sanitize@url \@href}%
\providecommand \@href[1]{\@@startlink{#1}\@@href}%
\providecommand \@@href[1]{\endgroup#1\@@endlink}%
\providecommand \@sanitize@url [0]{\catcode `\\12\catcode `\$12\catcode
  `\&12\catcode `\#12\catcode `\^12\catcode `\_12\catcode `\%12\relax}%
\providecommand \@@startlink[1]{}%
\providecommand \@@endlink[0]{}%
\providecommand \url  [0]{\begingroup\@sanitize@url \@url }%
\providecommand \@url [1]{\endgroup\@href {#1}{\urlprefix }}%
\providecommand \urlprefix  [0]{URL }%
\providecommand \Eprint [0]{\href }%
\providecommand \doibase [0]{http://dx.doi.org/}%
\providecommand \selectlanguage [0]{\@gobble}%
\providecommand \bibinfo  [0]{\@secondoftwo}%
\providecommand \bibfield  [0]{\@secondoftwo}%
\providecommand \translation [1]{[#1]}%
\providecommand \BibitemOpen [0]{}%
\providecommand \bibitemStop [0]{}%
\providecommand \bibitemNoStop [0]{.\EOS\space}%
\providecommand \EOS [0]{\spacefactor3000\relax}%
\providecommand \BibitemShut  [1]{\csname bibitem#1\endcsname}%
\let\auto@bib@innerbib\@empty
%</preamble>
\bibitem [{\citenamefont {Laneman}\ and\ \citenamefont
  {Wornell}(2000)}]{laneman2000energy}%
  \BibitemOpen
  \bibfield  {author} {\bibinfo {author} {\bibfnamefont {J.~N.}\ \bibnamefont
  {Laneman}}\ and\ \bibinfo {author} {\bibfnamefont {G.~W.}\ \bibnamefont
  {Wornell}},\ }\bibfield  {title} {\emph {\bibinfo {title} {\emph
  {Energy-efficient antenna sharing and relaying for wireless networks}},\
  }}in\ \href {\doibase 10.1109/WCNC.2000.904590} {\emph {\bibinfo {booktitle}
  {2000 IEEE Wireless Communications and Networking Conference. Conference
  Record (Cat. No.00TH8540)}}},\ Vol.~\bibinfo {volume} {1}\ (\bibinfo {year}
  {2000})\ pp.\ \bibinfo {pages} {7--12}\BibitemShut {NoStop}%
\bibitem [{\citenamefont {Salo}\ \emph {et~al.}(2006)\citenamefont {Salo},
  \citenamefont {El-Sallabi},\ and\ \citenamefont {Vainikainen}}]{4012470}%
  \BibitemOpen
  \bibfield  {author} {\bibinfo {author} {\bibfnamefont {J.}~\bibnamefont
  {Salo}}, \bibinfo {author} {\bibfnamefont {H.~M.}\ \bibnamefont
  {El-Sallabi}}, \ and\ \bibinfo {author} {\bibfnamefont {P.}~\bibnamefont
  {Vainikainen}},\ }\bibinfo {title} {\emph {Statistical Analysis of the
  Multiple Scattering Radio Channel}},\ \href {\doibase
  10.1109/TAP.2006.883964} {\bibfield  {journal} {\bibinfo  {journal} {IEEE
  Trans. Antennas Propag.}\ }\textbf {\bibinfo {volume} {54}},\ \bibinfo
  {pages} {3114} (\bibinfo {year} {2006})}\BibitemShut {NoStop}%
\bibitem [{\citenamefont {Karagiannidis}\ \emph {et~al.}(2007)\citenamefont
  {Karagiannidis}, \citenamefont {Sagias},\ and\ \citenamefont
  {Mathiopoulos}}]{4291825}%
  \BibitemOpen
  \bibfield  {author} {\bibinfo {author} {\bibfnamefont {G.~K.}\ \bibnamefont
  {Karagiannidis}}, \bibinfo {author} {\bibfnamefont {N.~C.}\ \bibnamefont
  {Sagias}}, \ and\ \bibinfo {author} {\bibfnamefont {P.~T.}\ \bibnamefont
  {Mathiopoulos}},\ }\bibinfo {title} {\emph {{$N$*N}akagami: A Novel
  Stochastic Model for Cascaded Fading Channels}},\ \href {\doibase
  10.1109/TCOMM.2007.902497} {\bibfield  {journal} {\bibinfo  {journal} {IEEE
  Trans. Commun.}\ }\textbf {\bibinfo {volume} {55}},\ \bibinfo {pages} {1453}
  (\bibinfo {year} {2007})}\BibitemShut {NoStop}%
\bibitem [{\citenamefont {Stojanac}\ \emph {et~al.}(2017)\citenamefont
  {Stojanac}, \citenamefont {Suess},\ and\ \citenamefont
  {Kliesch}}]{SPIE_Gauss}%
  \BibitemOpen
  \bibfield  {author} {\bibinfo {author} {\bibfnamefont {{\v{Z}}.}~\bibnamefont
  {Stojanac}}, \bibinfo {author} {\bibfnamefont {D.}~\bibnamefont {Suess}}, \
  and\ \bibinfo {author} {\bibfnamefont {M.}~\bibnamefont {Kliesch}},\
  }\bibfield  {title} {\emph {\bibinfo {title} {\emph {{On the distribution of
  a product of $N$ Gaussian random variables}}},\ }}in\ \href {\doibase
  10.1117/12.2275547} {\emph {\bibinfo {booktitle} {Proc.\ SPIE, Wavelets and
  Sparsity XVII}}},\ Vol.\ \bibinfo {volume} {10394}\ (\bibinfo {year}
  {2017})\BibitemShut {NoStop}%
\bibitem [{\citenamefont {Springer}\ and\ \citenamefont
  {Thompson}(1966)}]{siap/springer66}%
  \BibitemOpen
  \bibfield  {author} {\bibinfo {author} {\bibfnamefont {M.~D.}\ \bibnamefont
  {Springer}}\ and\ \bibinfo {author} {\bibfnamefont {W.~E.}\ \bibnamefont
  {Thompson}},\ }\bibinfo {title} {\emph {The distribution of products of
  independent random variables}},\ \href@noop {} {\bibfield  {journal}
  {\bibinfo  {journal} {SIAM J. Appl. Math.}\ }\textbf {\bibinfo {volume}
  {14}},\ \bibinfo {pages} {511} (\bibinfo {year} {1966})}\BibitemShut
  {NoStop}%
\bibitem [{\citenamefont {{Gaunt}}()}]{Products_Gaunt}%
  \BibitemOpen
  \bibfield  {author} {\bibinfo {author} {\bibfnamefont {R.~E.}\ \bibnamefont
  {{Gaunt}}},\ }\bibinfo {title} {\emph {Products of normal, beta and gamma
  random variables: {Stein} operators and distributional theory}},\ \href@noop
  {} {\bibfield  {journal} {\bibinfo  {journal} {Brazilian Journal of
  Probability and Statistics}\ }\textbf {\bibinfo {volume} {32}},\ \bibinfo
  {pages} {437}}\BibitemShut {NoStop}%
\bibitem [{\citenamefont {Gaunt}(2017)}]{Gau17}%
  \BibitemOpen
  \bibfield  {author} {\bibinfo {author} {\bibfnamefont {R.~E.}\ \bibnamefont
  {Gaunt}},\ }\bibinfo {title} {\emph {On Stein’s method for products of
  normal random variables and zero bias couplings}},\ \href {\doibase
  10.3150/16-BEJ848} {\bibfield  {journal} {\bibinfo  {journal} {Bernoulli}\
  }\textbf {\bibinfo {volume} {23}},\ \bibinfo {pages} {3311} (\bibinfo {year}
  {2017})},\ \Eprint {http://arxiv.org/abs/1309.4344} {arXiv:1309.4344
  [math.PR]}\BibitemShut {NoStop}%
\bibitem [{\citenamefont {{Gaunt}}\ \emph {et~al.}(2016)\citenamefont
  {{Gaunt}}, \citenamefont {Mijoule},\ and\ \citenamefont
  {Swan}}]{Stein_Op_Gaunt}%
  \BibitemOpen
  \bibfield  {author} {\bibinfo {author} {\bibfnamefont {R.~E.}\ \bibnamefont
  {{Gaunt}}}, \bibinfo {author} {\bibfnamefont {G.}~\bibnamefont {Mijoule}}, \
  and\ \bibinfo {author} {\bibfnamefont {Y.}~\bibnamefont {Swan}},\ }\bibinfo
  {title} {\emph {{Stein} operators for product distributions}},\ \href@noop {}
  {\  (\bibinfo {year} {2016})},\ \Eprint
  {http://arxiv.org/abs/arXiv:1604.06819} {arXiv:1604.06819}\BibitemShut
  {NoStop}%
\bibitem [{\citenamefont {Stoyanov}\ \emph {et~al.}(2014)\citenamefont
  {Stoyanov}, \citenamefont {Lin},\ and\ \citenamefont
  {DasGupta}}]{Stoyanov2014}%
  \BibitemOpen
  \bibfield  {author} {\bibinfo {author} {\bibfnamefont {J.}~\bibnamefont
  {Stoyanov}}, \bibinfo {author} {\bibfnamefont {G.~D.}\ \bibnamefont {Lin}}, \
  and\ \bibinfo {author} {\bibfnamefont {A.}~\bibnamefont {DasGupta}},\
  }\bibinfo {title} {\emph {Hamburger moment problem for powers and products of
  random variables}},\ \href {\doibase 10.1016/j.jspi.2013.11.002} {\bibfield
  {journal} {\bibinfo  {journal} {Journal of Statistical Planning and
  Inference}\ }\textbf {\bibinfo {volume} {154}},\ \bibinfo {pages} {166}
  (\bibinfo {year} {2014})}\BibitemShut {NoStop}%
\bibitem [{\citenamefont {Springer}\ and\ \citenamefont
  {Thompson}(1970)}]{Distribution_of_Products_70}%
  \BibitemOpen
  \bibfield  {author} {\bibinfo {author} {\bibfnamefont {M.~D.}\ \bibnamefont
  {Springer}}\ and\ \bibinfo {author} {\bibfnamefont {W.~E.}\ \bibnamefont
  {Thompson}},\ }\bibinfo {title} {\emph {{The distribution of products of
  beta, gamma and Gaussian random variables}}},\ \href {\doibase
  10.1137/0118065} {\bibfield  {journal} {\bibinfo  {journal} {SIAM J. Appl.
  Math.}\ }\textbf {\bibinfo {volume} {18}},\ \bibinfo {pages} {721} (\bibinfo
  {year} {1970})}\BibitemShut {NoStop}%
\bibitem [{\citenamefont {Kilbas}\ and\ \citenamefont
  {Saigo}(2004)}]{Htransform}%
  \BibitemOpen
  \bibfield  {author} {\bibinfo {author} {\bibfnamefont {A.~A.}\ \bibnamefont
  {Kilbas}}\ and\ \bibinfo {author} {\bibfnamefont {M.}~\bibnamefont {Saigo}},\
  }\href@noop {} {\emph {\bibinfo {title} {H-Transforms: Theory and
  Applications}}},\ Analytical Methods and Special Functions\ (\bibinfo
  {publisher} {Chapman \& Hall/CRC},\ \bibinfo {address} {Boca Raton},\
  \bibinfo {year} {2004})\ Chap.~\bibinfo {chapter} {1}\BibitemShut {NoStop}%
\bibitem [{\citenamefont {Luke}(1969)}]{Luke}%
  \BibitemOpen
  \bibfield  {author} {\bibinfo {author} {\bibfnamefont {Y.~L.}\ \bibnamefont
  {Luke}},\ }\href@noop {} {\emph {\bibinfo {title} {{The special functions and
  their approximations [Volume II]}}}}\ (\bibinfo  {publisher} {Academic
  Press},\ \bibinfo {address} {New York},\ \bibinfo {year} {1969})\BibitemShut
  {NoStop}%
\bibitem [{\citenamefont {Cook~Jr}(1981)}]{Coo81}%
  \BibitemOpen
  \bibfield  {author} {\bibinfo {author} {\bibfnamefont {I.~D.}\ \bibnamefont
  {Cook~Jr}},\ }\href {http://www.dtic.mil/docs/citations/ADA106768} {\emph
  {\bibinfo {title} {The {H}-function and probability density functions of
  certain algebraic combinations of independent random variables with
  {H}-function probability distribution}}},\ \bibinfo {type} {Tech. Rep.}\
  (\bibinfo  {institution} {Air Force Inst. of Tech. Wright-Patterson AFB OH},\
  \bibinfo {year} {1981})\BibitemShut {NoStop}%
\bibitem [{\citenamefont {Bateman}(1953)}]{bateman1953higher}%
  \BibitemOpen
  \bibfield  {author} {\bibinfo {author} {\bibfnamefont {H.}~\bibnamefont
  {Bateman}},\ }\href@noop {} {\emph {\bibinfo {title} {{Higher Transcendental
  Functions [Volume I]}}}}\ (\bibinfo  {publisher} {McGraw-Hill Book Company},\
  \bibinfo {address} {New York},\ \bibinfo {year} {1953})\BibitemShut {NoStop}%
\bibitem [{\citenamefont {Bateman}\ and\ \citenamefont
  {Erdelyi}(1954)}]{bateman1954tables}%
  \BibitemOpen
  \bibfield  {author} {\bibinfo {author} {\bibfnamefont {H.}~\bibnamefont
  {Bateman}}\ and\ \bibinfo {author} {\bibfnamefont {A.}~\bibnamefont
  {Erdelyi}},\ }\href@noop {} {\emph {\bibinfo {title} {{Tables of Integral
  Transforms [Volume II]}}}}\ (\bibinfo  {publisher} {Mc Graw Hill},\ \bibinfo
  {address} {New York},\ \bibinfo {year} {1954})\BibitemShut {NoStop}%
\bibitem [{\citenamefont {Askey}\ and\ \citenamefont
  {Daalhuis}(2010)}]{NISTmathfun:2010}%
  \BibitemOpen
  \bibfield  {author} {\bibinfo {author} {\bibfnamefont {R.~A.}\ \bibnamefont
  {Askey}}\ and\ \bibinfo {author} {\bibfnamefont {A.~B.~O.}\ \bibnamefont
  {Daalhuis}},\ }in\ \href@noop {} {\emph {\bibinfo {booktitle} {NIST Handbook
  of Mathematical Functions}}},\ \bibinfo {editor} {edited by\ \bibinfo
  {editor} {\bibfnamefont {F.~W.~J.}\ \bibnamefont {Olver}}, \bibinfo {editor}
  {\bibfnamefont {D.~W.}\ \bibnamefont {Lozier}}, \bibinfo {editor}
  {\bibfnamefont {R.~F.}\ \bibnamefont {Boisvert}}, \ and\ \bibinfo {editor}
  {\bibfnamefont {C.~W.}\ \bibnamefont {Clark}}}\ (\bibinfo  {publisher}
  {Cambridge University Press},\ \bibinfo {address} {New York},\ \bibinfo
  {year} {2010})\ Chap.~\bibinfo {chapter} {16}, pp.\ \bibinfo {pages}
  {403--418}\BibitemShut {NoStop}%
\end{thebibliography}%

%%%%%%%%%%%%%%%%%%%%%%
%%%%%%%%%%%%%%%%%%%%%%
%%%%%%%%%%%%%%%%%%%%%%
% \newpage
% \appendices
\appendix
\section*{Appendices}
\addcontentsline{toc}{section}{Appendices}
\renewcommand{\thesubsection}{\Alph{subsection}}

In the appendices we define the Meijer G-functions and state some of their basic properties (Appendix~\ref{sec:intro_meijer_g}) and prove several lemmas used throughout the paper, as well as provide the proof of  Proposition~\ref{prop:CDF} (Appendix~\ref{App:Proofs}).
% ---------------------------------------

\subsection{Introduction to Meijer G-functions}
\label{sec:intro_meijer_g}
The results of this article rely on the theory of Meijer G-functions.
In this appendix, we introduce these functions together with some of their properties.
All results presented in this appendix can be found in Refs.~\cite{bateman1953higher,bateman1954tables,NISTmathfun:2010}.

Meijer G-functions are a family of special functions in one variable that is closed under several operations including
\begin{equation*}
x \mapsto -x, \,
x \mapsto 1/x,
\text{ multiplication by } x^p,
\text{ differentiation},
\text{ and integration}.
\end{equation*}
\begin{definition}\label{Def:MeijerG}
For integers $m,n,p,q$ satisfying $0\leq m \leq q$, $0 \leq n \leq p$ and for  numbers $a_i, b_j \in \mathbb{C}$ (with $i=1,\ldots,p$; $j=1,\ldots,q$), the Meijer G-function
$G_{p,q}^{m,n}\roundbra{\, \cdot \,\, \Big| \begin{matrix} a_1,a_2, \ldots a_p \\ b_1, b_2,\ldots, b_q\end{matrix}}$ is defined by the line integral
\begin{equation}
\label{eq:the_G}
G_{p,q}^{m,n}\roundbra{z \Big| \begin{matrix} a_1,a_2, \ldots a_p \\ b_1, b_2,\ldots, b_q\end{matrix}}
=
 \frac{1}{2\pi \i} \int_{\mathcal{L}} \mathcal{H}_{p,q}^{m,n} (s) z^{-s} \, \d s \, ,
\end{equation}
with
\begin{equation}\label{Hmnpq}
 \mathcal{H}_{p,q}^{m,n}(s) \coloneqq   \frac{\prod_{j=1}^m \Gamma(b_j+s) \prod_{i=1}^n \Gamma(1-a_i-s)}{\prod_{i=n+1}^p \Gamma(a_i+s) \prod_{j=m+1}^q \Gamma(1-b_j-s)}
 \,.
\end{equation}
Here,
\begin{equation}
z^{-s}=\exp\roundbra{-s \curlybra{\log |z|+ \i \arg z}}, \quad z \neq 0, \quad \i=\sqrt{-1},
\end{equation}
where $\log|z|$ represents the natural logarithm of $|z|$ and $\arg z$ is not necessarily the principal value.
Empty products are identified with one.
The parameter vectors $a$ and $b$ need to be chosen such that the poles
\begin{equation}\label{polesB}
b_{j\ell} = -b_j - \ell \quad (j=1,2,\ldots,m;\, \ell=0,1,2,\ldots)
\end{equation}
of the gamma functions $s\mapsto \Gamma(b_j+s)$ and the poles
\begin{equation}\label{polesA}
a_{ik} = 1-a_i+k \quad (i=1,2,\ldots,n;\, k=0,1,2,\ldots)
\end{equation}
of the gamma functions $s\mapsto \Gamma(1-a_i-s)$ do not coincide, i.e.
\begin{equation}
b_j+\ell \neq a_i-k-1 \quad (i=1,\ldots,n;\, j=1,\ldots,m;\,  k,\ell=0,1,2,\ldots).
\end{equation}
The integral is taken over an infinite contour $\mathcal{L}$ that separates all  poles $b_{j\ell}$ in \eqref{polesB} to the left and $a_{ik}$ in \eqref{polesA} to the right of $\mathcal{L}$, and has one of the following forms:
\begin{enumerate}
\item $\mathcal{L}=\mathcal{L}_{-\infty}$ is a left loop situated in a horizontal strip starting at the point $-\infty+\i \phi_1$ and terminating at the point $-\infty + \i \phi_2$ with $-\infty < \phi_1 < \phi_2 < +\infty$;
\item $\mathcal{L}=\mathcal{L}_{+\infty}$ is a right loop situated in a horizontal strip starting at the point $+\infty+\i \phi_1$ and terminating at the point $+\infty + \i \phi_2$ with $-\infty < \phi_1 < \phi_2 < +\infty$;
\item $\mathcal{L}=\mathcal{L}_{\i \gamma \infty}$ is a contour starting at the point $\gamma-\i \infty$ and terminating at the point $\gamma+\i \infty$, where $\gamma \in \R$.
\end{enumerate}
\end{definition}

In this work, we exploit the following properties of Meijer G-functions:
\begin{itemize}
\item Inverse of the argument
\begin{equation}\label{reverseG}
 G_{p,q}^{m,n}  \left(z \Big|  \begin{matrix} a_1,\ldots,a_p \\ b_1,\ldots,b_q \end{matrix} \right) = G_{q,p}^{n,m}  \left(z^{-1} \Big|  \begin{matrix}  1-b_1,\ldots,1- b_q  \\ 1-a_1,\ldots,1-a_p \end{matrix} \right)
\end{equation}
\item Product with monomials
\begin{equation}\label{zrhoG}
z^{\rho} G_{p,q}^{m,n}  \left(z \Big|  \begin{matrix} a_1,\ldots,a_p \\ b_1,\ldots,b_q \end{matrix} \right) = G_{p,q}^{m,n}  \left(z \Big|  \begin{matrix} a_1+\rho,\ldots,a_p+\rho \\ b_1+\rho,\ldots,b_q+\rho \end{matrix} \right)
\end{equation}

\item Derivative of a Meijer G-function
\begin{align}\label{MGder}
z \frac{d}{dz} G_{p,q}^{m,n}  &\left(z \Big|  \begin{matrix} a_1,\ldots,a_p \\ b_1,\ldots,b_{q} \end{matrix} \right)  \nonumber \\
&=  G_{p,q}^{m,n}  \left(z \Big|  \begin{matrix} a_1-1,a_2,\ldots,a_p \\ b_1,\ldots,b_q \end{matrix} \right) %\nonumber\\ &
+ (a_1-1) G_{p,q}^{m,n}  \left(z \Big|  \begin{matrix}a_1,\ldots,a_p \\ b_1,\ldots,b_q \end{matrix} \right)
\end{align}

\item Integration of a Meijer G-function multiplied by certain polynomials
\begin{align}\label{int1xx-1G}
\int_1^{\infty} x^{-\rho} (x-1)^{\sigma-1} G_{p,q}^{m,n} & \left(\alpha x \Big|  \begin{matrix} a_1,\ldots,a_p \\ b_1,\ldots,b_q \end{matrix} \right) \, \d x  \nonumber \\
&= \Gamma(\sigma) G_{p+1,q+1}^{m+1,n}  \left(\alpha \Big|  \begin{matrix} a_1,\ldots,a_p,\rho \\ \rho-\sigma,b_1,\ldots,b_q \end{matrix} \right)
\end{align}
with conditions of validity
\begin{equation} \label{int1xx-1Gvalidity}
\begin{matrix}
& p+q < 2(m+n), & |\arg \alpha| < \roundbra{m+n - p/2 - q/2 } \pi \\
& \Re(\rho -\sigma- a_j)>-1, \quad j=1,\ldots,n, & \Re \sigma >0
\end{matrix}
\end{equation}

\item Integration of a Meijer G-function multiplied by the exponential function and a monomial
\begin{equation}\label{int0xeG}
\int_0^{\infty} x^{-\rho} \e^{-\beta x} G_{p,q}^{m,n}  \left(\alpha x \Big|  \begin{matrix} a_1,\ldots,a_p \\ b_1,\ldots,b_q \end{matrix} \right) \, \d x = \beta^{\rho-1} G_{p+1,q}^{m,n+1}  \left(\frac{\alpha}{\beta} \Big|  \begin{matrix} \rho,a_1,\ldots,a_p \\ b_1,\ldots,b_q \end{matrix} \right)
\end{equation}
with conditions of validity
\begin{equation}\label{eq:cond_v}
\begin{matrix}
& p+q < 2(m+n), & |\arg \alpha| <  \pi \roundbra{m+n - p/2 - q/2} \\
& |\arg \beta| < \pi/2             & \Re(b_j-\rho)>-1, \quad j=1,\ldots,m
\end{matrix}
\end{equation}

\item Integration of a Meijer G-function multiplied by an exponential function
\begin{align}\label{int0eG}
\int_0^{\infty} \e^{-\beta x} G_{p,q}^{m,n}  &\left(\alpha x^2 \Big|  \begin{matrix} a_1,\ldots,a_p \\ b_1,\ldots,b_q \end{matrix} \right) \, \d x \nonumber \\
&= \pi^{-1/2} \beta^{-1} G_{p+2,q}^{m,n+2}  \left(\frac{4\alpha}{\beta^2} \Big|  \begin{matrix} 0,1/2,a_1,\ldots,a_p \\ b_1,\ldots,b_q \end{matrix} \right)
\end{align}
with conditions of validity
\begin{equation} \label{int0eGvalidity}
\begin{matrix}
& p+q < 2(m+n), & |\arg \alpha| < \roundbra{m+n -p/2-q/2 } \pi \\
& |\arg \beta| < \pi/2             & \Re(b_j)>-1/2, \quad j=1,\ldots,m
\end{matrix}
\end{equation}
\end{itemize}

\subsection{Proofs of Lemmas}\label{App:Proofs}
In this section we present the proofs of several lemmas introduced previously in the main text.
We start by proving Proposition~\ref{prop:CDF}, i.e.\ the following special case of that statement (the rest has been shown previously, immediately after stating the lemma).

\begin{lemma}\label{lemma:CDFofY}
Let $\{g_i\}_{i=1}^N $ be a set of iid standard Gaussian random variables (i.e., $g_i \sim \mathcal{N}(0,1)$) and $Y \coloneqq \prod_{i=1}^N g_i^2$. Then, for any $t> 0$,
\begin{align}
\PP\roundbra{Y \leq t} &=
1 -  \frac{1}{\pi^{\frac{N}{2}}}    G_{N+1, 1}^{0,N+1} \left(\frac{2^N}{t} \Big|  \begin{matrix} 1,1/2,\ldots,1/2 \\ 0 \end{matrix} \right).
\end{align}
\end{lemma}
\begin{proof}%[of Lemma \ref{thm:CDF}]
Notice the following observation (with $X\coloneqq \prod_{i=1}^N g_i$)
\begin{equation}
\begin{aligned}
\PP\left(Y \leq t\right)&=\PP \left(X^2 \leq t\right)
= 
1 - \PP \left(X^2 \geq t\right) 
\\
&=1- \PP \left(X \geq \sqrt{t}\right) - \PP \left(X \leq -\sqrt{t}\right) \\
&=1- \int_{\sqrt{t}}^{\infty} f_X(x) \,\d x - \int_{-\infty}^{-\sqrt{t}} f_X(x) \,\d x \\
&=1- \int_{\sqrt{t}}^{\infty} f_X(x) \,\d x - \int_{\sqrt{t}}^{\infty} f_X(-x) \,\d x ,
\end{aligned}
\end{equation}
where $f_X$ denotes the probability density function (PDF) of the random variable $X$.
So, it is enough to consider the random variable $X=\prod_{i=1}^N g_i$, where $\{g_i\}_{i=1}^N$ are iid standard Gaussian random variables.
It is well-known that the PDF of $X$
is given by
\begin{equation}\label{pdf_product_gaussians}
f_X(x) = \frac{1}{(2\pi)^{N/2}} G_{0,N}^{N,0} \left(\frac{x^2}{2^N} \Big|  0 \right), \quad \text{for } x \in \mathbb{R} \, ,
\end{equation}
where $G$ denotes the Meijer G-function, see~\cite{Distribution_of_Products_70, Products_Gaunt}
(here: $m=N=q$, $n=0=p$, $a = \emptyset$, $b = (0,\dots, 0)$, $z = x^2\,2^{-N}$).
That is, (since $f_X(x)$ is an even function)
\begin{align}
\PP \left(Y \leq t\right)=\PP \left(X^2 \leq t\right)& = 1- 2\int_{\sqrt{t}}^{\infty} f_X(x) \,dx   = 1 - 2\int_{\sqrt{t}}^{\infty}\frac{1}{(2\pi)^{N/2}} G_{0,N}^{N,0} \left(\frac{x^2}{2^N} \Big|  0 \right) \, \d x
\nonumber \\
& =1- \frac{2}{(2\pi)^{N/2}} \int_{\sqrt{t}}^{\infty} G_{0,N}^{N,0} \left(\frac{x^2}{2^N} \Big|  0 \right) \,dx . \label{eq:PY<=t}
\end{align}
In the following we compute the integral
\begin{equation}
\begin{aligned}
\int_{\sqrt{t}}^{\infty} G_{0,N}^{N,0} \left(\frac{x^2}{2^N} \Big| 0 \right) \,\d x 
&=
\begin{cases} 
  v=\frac{x^2}{t} \, & \text{for } x=\sqrt{t} \,\Rightarrow v=1 
  \\ 
  \d v=\frac{2}{t}x \, \d x & x=\sqrt{tv}
\end{cases} 
\\
&=  \int_{1}^{\infty} G_{0,N}^{N,0} \left(\frac{t}{2^N} v \Big| 0 \right) \frac{t}{2} t^{-1/2} v^{-1/2} \, \d v
\\
 &=  \frac{\sqrt{t}}{2} \int_{1}^{\infty}  v^{-1/2} G_{0,N}^{N,0} \left(\frac{t}{2^N} v \Big| 0 \right)  \, \d v   .
\end{aligned}
\end{equation}
Following the notation in \eqref{int1xx-1G} we have
\begin{equation}
\begin{matrix}
\rho=1/2 & m=N & n=0 & \alpha=\frac{t}{2^N} & b_i=0 \\
\sigma=1 & p=0 & q=N. & &
\end{matrix}
\end{equation}
In order to apply the result we have to check the set of conditions of validity \eqref{int1xx-1Gvalidity}:
\begin{equation}
\begin{matrix}
0+N=p+q < 2(m+n)= 2(0+N) \\
 0=|\arg \alpha| < (m+n -p/2 -q/2) \pi = (N+0-0-N/2) \pi \\
\Re(\rho -\sigma-a_j) >-1, \quad j=1,\ldots,n, \\
1=\Re \sigma>0 .
\end{matrix}
\end{equation}
Since the conditions are satisfied, we obtain that
\begin{equation}
\begin{aligned}
\int_{\sqrt{t}}^{\infty} G_{0,N}^{N,0} \left(\frac{x^2}{2^N} \Big|  0 \right) \,dx & = \frac{\sqrt{t}}{2}\underbrace{\Gamma(1)}_{=1} G_{1, N+1}^{N+1,0} \left(\frac{t}{2^N} \Big|  \begin{matrix} 1/2 \\ -1/2, 0,\ldots, 0 \end{matrix} \right) \\
&= \frac{\sqrt{t}}{2} G_{1, N+1}^{N+1,0} \left(\frac{t}{2^N} \Big|  \begin{matrix} 1/2 \\ -1/2, 0,\ldots, 0 \end{matrix} \right).
\end{aligned}
\end{equation}
Thus, continuing the estimate \eqref{eq:PY<=t} we obtain
\begin{equation}
\begin{aligned}
\PP \left(Y \leq t\right)
& =1- \frac{2}{(2\pi)^{N/2}} \int_{\sqrt{t}}^{\infty} G_{0,N}^{N,0} \left(\frac{x^2}{2^N} \left| \right. 0 \right) \, \d x  \\
& =1- \frac{\sqrt{t}}{(2\pi)^{N/2}} G_{1, N+1}^{N+1,0} \left(\frac{t}{2^N} \Big|  \begin{matrix} 1/2 \\ -1/2, 0,\ldots, 0 \end{matrix} \right).
\end{aligned}
\end{equation}
Now using property \eqref{zrhoG} of Meijer G-function  leads to %(see Page 1034 in \cite{Tables_of_Integrals_Gradshteyn}, also Wikipedia)
to obtain
\begin{equation}
\PP \left(Y \leq t\right) = 1- \frac{1}{\pi^{N/2}} G_{1, N+1}^{N+1,0} \left(\frac{t}{2^N} \Big|  \begin{matrix} 1 \\ 0, 1/2,\ldots, 1/2 \end{matrix} \right).
\end{equation}
The claim now follows from \eqref{reverseG}.
\end{proof}

To prove Lemma~\ref{lemma:limfder} we use the following result.
\begin{lemma}\label{lemma:fder}
Define $f(s) \coloneqq (s-1/2-k) \Gamma(1/2-s)$. Then for every $j \in \N_0$ it holds that
\begin{equation}
f^{(j)}(s)=(-1)^{j-1} \sqbra{j \Gamma^{(j-1)}(1/2-s) - \Gamma^{(j)}(1/2-s)(s-1/2-k)}  .
\end{equation}
\end{lemma}

\begin{proof}
We prove the above lemma by induction. For $j=0$, the statement is clearly true. Assume now that for $j-1$ it holds that
\begin{equation}
f^{(j-1)}(s)=(-1)^{j-2} \sqbra{(j-1) \Gamma^{(j-2)}(1/2-s) - \Gamma^{(j-1)}(1/2-s)(s-1/2-k)}.
\end{equation}
Then in the $j$th step we obtain
\begin{equation}
\begin{aligned}
&f^{(j)}(s) = \frac{d}{ds} f^{(j-1)}(s)  \\
&= (-1)^{j-2} \left[(j-1) \Gamma^{(j-1)}(1/2-s)\cdot(-1) \right. \\
& \quad \quad \quad \quad \left. - \roundbra{\Gamma^{(j)}(1/2-s) \cdot (-1) (s-1/2-k) + \Gamma^{(j-1)}(1/2-s)}\right] \\
&= (-1)^{j-1} \sqbra{(j-1) \Gamma^{(j-1)}(1/2-s) - \Gamma^{(j)}(1/2-s) (s-1/2-k) - \Gamma^{(j-1)}(1/2-s)} \\
&= (-1)^{j-1} \sqbra{j \Gamma^{(j-1)}(1/2-s) - \Gamma^{(j)}(1/2-s) (s-1/2-k) } ,
\end{aligned}
\end{equation}
which finishes the proof.
\end{proof}

Now we are ready to prove Lemma~\ref{lemma:limfder}.
\begin{proof}[Proof of Lemma~\ref{lemma:limfder}]
Classical results in basic complex analysis imply that a power series of the Gamma function at the simple pole $z_0$ (i.e.\ at pole $z_0$  of multiplicity one) is of the form
\begin{equation}
\Gamma(z)=\frac{1}{z-z_0} g(z),
\end{equation}
where $g$ is analytic at $z_0$ and thus can be expanded into the Taylor series around $z_0$. That is,
\begin{equation}
g(z)=(z-z_0) \Gamma(z) = \sum_{\ell=0}^{\infty} \frac{g^{(\ell)}(z_0)}{\ell!} (z-z_0)^{\ell}.
\end{equation}
Rearranging the above equation results in
\begin{equation}
\Gamma(z)= \frac{1}{z-z_0} \sum_{\ell=0}^{\infty} \frac{g^{(\ell)}(z_0)}{\ell!} (z-z_0)^{\ell}.
\end{equation}
Now, recall that the Gamma function has simple poles at $z_0=-k$, with $k \in \N_0$. Hence, we have
\begin{equation}\label{eq:GammaTaylor}
\Gamma(z)=\frac{1}{z+k} \sum_{\ell=0}^{\infty} \frac{g^{(\ell)}(-k)}{\ell!} (z+k)^{\ell},
\end{equation}
where
\begin{equation}
g(z)=(z+k) \Gamma(z) =\frac{\Gamma(z+k+1)}{ z(z+1)\cdots (z+k-1)} = \prod_{i=1}^{k+1} g_i(z)
\end{equation}
with
\begin{equation}
\begin{aligned}
& g_i(z)\coloneqq  (z+i-1)^{-1}\quad \text{ for all } i \in \sqbra{k} \\
& g_{k+1}(z)\coloneqq \Gamma(z+k+1).
\end{aligned}
\end{equation}
By induction it can be shown that the corresponding $\ell$-th derivatives (for $\ell \in \N_0$) are of the form
\begin{equation}
\begin{aligned}
& g_i^{\ell}(z)=(-1)^{\ell} \, \ell ! \, (z+i-1)^{-(\ell+1)}\quad \text{ for all } i \in \sqbra{k} \\
& g_{k+1}^{\ell}(z)=\Gamma^{(\ell)}(z+k+1).
\end{aligned}
\end{equation}
Applying the general Leibniz's rule leads to
\begin{equation}
\begin{aligned}
& g^{(\ell)}(z) = \roundbra{\prod_{i=1}^{k+1} g_i(z)}^{(\ell)} = \sum_{\ell_1+\ell_2+\ldots+\ell_{k+1}=\ell} \binom{\ell}{\ell_1,\ell_2,\ldots,\ell_{k+1}} \prod_{i=1}^{k+1} g_i^{(\ell_i)}(z) \\
& = \sum_{\ell_1+\ldots+\ell_{k+1}=\ell} \frac{\ell!}{ \ell_1! \, \cdots \, \ell_{k+1}!} \roundbra{\prod_{i=1}^k (-1)^{\ell_i} \, \ell_i!\, \roundbra{z+i-1}^{-(\ell_i+1)}} \Gamma^{(\ell_{k+1})} (z+k+1) \\
& = \sum_{\ell_1+\ldots+\ell_{k+1}=\ell} \frac{\ell!}{\ell_{k+1}!} \roundbra{\prod_{i=1}^k (-1)^{\ell_i} \cdot (-1)^{-(\ell_i+1)} \, \roundbra{-z-i+1}^{-(\ell_i+1)}} \Gamma^{(\ell_{k+1})} (z+k+1) \\
& = \sum_{\ell_1+\ldots+\ell_{k+1}=\ell} \frac{\ell!}{\ell_{k+1}!} (-1)^{-k} \roundbra{\prod_{i=1}^k  \roundbra{-z-i+1}^{-(\ell_i+1)}} \Gamma^{(\ell_{k+1})} (z+k+1) .
\end{aligned}
\end{equation}
Setting $z=-k$ in the above equation results in
\begin{align}
g^{(\ell)}(-k) & = \sum_{\ell_1+\ell_2+\ldots+\ell_{k+1}=\ell} \frac{\ell!}{\ell_{k+1}!} (-1)^{-k} \roundbra{\prod_{i=1}^k  \roundbra{k-i+1}^{-(\ell_i+1)}} \Gamma^{(\ell_{k+1})} (1) 
\nonumber\\
& = (-1)^{k} \, \ell! \,  \sum_{\ell_1+\ell_2+\ldots+\ell_{k+1}=\ell}   \roundbra{\prod_{i=1}^k  \roundbra{k-i+1}^{-(\ell_i+1)}} \frac{\Gamma^{(\ell_{k+1})} (1)}{\ell_{k+1}!} .
\end{align}
Plugging it into \eqref{eq:GammaTaylor} leads to
\begin{equation}
\begin{aligned}
\Gamma(z) & = \frac{1}{z+k} \sum_{\ell=0}^{\infty} \frac{g^{(\ell)}(-k)}{ \ell!} (z+k)^{\ell} \\
& = \frac{1}{z+k} \sum_{\ell=0}^{\infty} \roundbra{ (-1)^{k} \, \ell! \,  \sum_{\ell_1+\ldots+\ell_{k+1}=\ell}   \roundbra{\prod_{i=1}^k  \roundbra{k-i+1}^{-(\ell_i+1)}} \frac{\Gamma^{(\ell_{k+1})} (1)}{\ell_{k+1}!}} \frac{(z+k)^{\ell}}{\ell!} \\
& = \frac{(-1)^{k} }{z+k} \sum_{\ell=0}^{\infty} \roundbra{ \sum_{\ell_1+\ldots+\ell_{k+1}=\ell}   \roundbra{\prod_{i=1}^k  \roundbra{k-i+1}^{-(\ell_i+1)}} \frac{\Gamma^{(\ell_{k+1})} (1)}{\ell_{k+1}!}} (z+k)^{\ell} \\
& = (-1)^{k} \sum_{\ell=0}^{\infty} c_{k,\ell} \, (z+k)^{\ell-1} \\
& = (-1)^k \curlybra{\frac{c_{k,0}}{z+k} + c_{k,1} + c_{k,2}(z+k) + \sum_{\ell=3}^{\infty} c_{k,\ell} (z+k)^{\ell-1}},
\end{aligned}
\end{equation}
where
\begin{equation}
c_{k,\ell} : =  \sum_{\ell_1+\ell_2+\ldots+\ell_{k+1}=\ell}   \roundbra{\prod_{i=1}^k  \roundbra{k-i+1}^{-(\ell_i+1)}} \frac{\Gamma^{(\ell_{k+1})} (1)}{\ell_{k+1}!}.
\end{equation}
Next, by induction one can show that the $j$th derivative of Gamma function (with $j \in \N_0$) is of the form
\begin{equation}\nonumber
\Gamma^{(j)}(z)  
=
 (-1)^k \left\{(-1)^j \, j! \, (z+k)^{-(j+1)} c_{k,0} + j! \, c_{k,j+1}
 + \sum_{\ell=j+2}^{\infty} c_{k,\ell} \prod_{\rho=1}^j (\ell-\rho) (z+k)^{\ell-j-1}\right\},
\end{equation}
with convention that an empty product $\prod_{\rho=1}^0=1$.
Finally, we are ready to compute the limit. By Lemma \ref{lemma:fder}, the above analysis, and using the substitution $\varepsilon\coloneqq 1/2+k-s$ we obtain for~$j \in \N$
\begin{equation}\nonumber
\begin{aligned}
&\lim_{s \rightarrow 1/2+k} f^{(j)}(s)  = \lim_{s\rightarrow 1/2+k} (-1)^{j-1} \sqbra{j \Gamma^{(j-1)}(1/2-s) - \Gamma^{(j)}(1/2-s)(s-1/2-k)} \\
& = (-1)^{j-1}  \lim_{\varepsilon \rightarrow 0} \curlybra{j \Gamma^{(j-1)}(-k+\varepsilon) + \varepsilon \Gamma^{(j)}(-k+\varepsilon)} \\
& = (-1)^{j-1}  \lim_{\varepsilon \rightarrow 0} j (-1)^k \curlybra{(-1)^{j-1} \,(j-1)! \, \varepsilon^{-j} c_{k,0} + (j-1)! \, c_{k,j} + \sum_{\ell=j+1}^{\infty} c_{k,\ell} \prod_{\rho=1}^{j-1}(\ell - \rho)   \varepsilon^{\ell-j} } \\
& \quad + (-1)^{j-1} \lim_{\varepsilon \rightarrow 0} \varepsilon (-1)^k \curlybra{(-1)^j \, j! \, \varepsilon^{-(j+1)} c_{k,0} + j! \, c_{k,j+1} + \sum_{\ell=j+2}^{\infty} c_{k,\ell} \prod_{\rho=1}^{j}(\ell -\rho) \varepsilon^{\ell-j-1}} \\
& = (-1)^{k+j-1} \lim_{\varepsilon \rightarrow 0} \left\{(-1)^{j-1} \,j! \, \varepsilon^{-j} c_{k,0} + j! \, c_{k,j} + j \,\sum_{\ell=j+1}^{\infty} c_{k,\ell} \prod_{\rho=1}^{j}(\ell - \rho) \varepsilon^{\ell-j}  \right. \\
&\qquad\qquad\qquad\qquad\qquad
\left. + (-1)^j \, j! \, \varepsilon^{-j} c_{k,0} + \varepsilon \, j! \, c_{k,j+1} + \sum_{\ell=j+2}^{\infty} c_{k,\ell} \prod_{\rho=1}^{j}(\ell -\rho) \varepsilon^{\ell-j} \right\} \\
& = (-1)^{k+j-1} \, j! \, c_{k,j} \\
& = (-1)^{k+j-1} \, j! \, \sum_{\ell_1+\ell_2+\ldots+\ell_{k+1}=j}   \roundbra{\prod_{i=1}^k  \roundbra{k-i+1}^{-(\ell_i+1)}} \frac{\Gamma^{(\ell_{k+1})} (1)}{\ell_{k+1}!} \, ,
\end{aligned}
\end{equation}
which finishes the proof.
\end{proof}

Next, we provide a proof of Lemma~\ref{lemma:sumYpowerlog}. The proof is analogous to the one of Theorem~\ref{thm:MeijerGseries}.
\begin{proof}[Proof of Lemma~\ref{lemma:sumYpowerlog}]\hfill\\
By~\cite[Theorem 1.2 (iii)]{Htransform}, the Meijer G-function 
$G_{N,1}^{1,N} \left(z \Big| \begin{matrix}  1/2, 1/2, \ldots, 1/2 \\ x \end{matrix} \right)$ with $x \in \{0,1\}$ 
is an analytic function of $z$ in the sector $\abs{\arg z}<\frac{(N+1)\pi}{2}$ (this is true for $N \geq 3$ since $\arg z=0$ for $z \in \R_0^+$ and $\arg z=\pi$, for $z \in \R^-$).
Therefore, we have
\begin{equation}
G_{N,1}^{1,N} \left(z \Big| \begin{matrix} 1/2, 1/2, \ldots, 1/2 \\ x \end{matrix} \right)= -\sum_{i=1}^{N}\sum_{k=0}^{\infty} \Res_{s=a_{ik}} \sqbra{{}_{x}\mathcal{H}_{N,1}^{1,N}(s) z^{-s}},
\end{equation}
where ${}_{x}\mathcal{H}_{N,1}^{1,N}$ is given by \eqref{Hmnpq} with $m=q=1$, $n=p=N$,   $a_1=\ldots=a_{N}=1/2$, and $b_1=x$, i.e.\
\begin{equation}
{}_{x}\mathcal{H}_{N,1}^{1,N} (s)  =  \Gamma(x+s) \Gamma^N(1/2-s)
\end{equation}
and $a_{ik}=1-a_i+k$ with $k \in \NN_0$ denote the poles of $\Gamma(1-a_i-s)$.

In our scenario we only have poles $a_{1k}=1/2+k$ of order $N$ with $k \in \N_0$.  Therefore,
\begin{equation}%\label{MeijerGinRes}
G_{N,1}^{1,N} \left(z \Big| \begin{matrix} 1/2, 1/2, \ldots 1/2 \\ x \end{matrix} \right)=  \sum_{k=0}^{\infty} -\Res_{s=a_{1k}} \sqbra{{}_{x}\mathcal{H}_{N,1}^{1,N}(s) z^{-s}}.
\end{equation}
%with
Since $a_{1k}=1/2+k$ is an $N$-th order pole of $\Gamma(1/2-s)$ for all $k \in \N_0$, also the integrand ${}_{x}\mathcal{H}_{N,1}^{1,N}(s)= \Gamma(x+s) \Gamma^N(1/2-s)$ has a pole $a_{1k}$ of order $N$ for all $k \in \N_0$.
Thus
\begin{equation}
\begin{aligned}
&\Res_{s=1/2+k} \sqbra{{}_{x}\mathcal{H}_{N,1}^{1,N}(s)z^{-s}} 
\\
& = \frac{1}{(N-1)!} \lim_{s \rightarrow 1/2+k} \sqbra{(s-1/2-k)^N {}_{x}\mathcal{H}_{N,1}^{1,N}(s) z^{-s}}^{(N-1)}
\\
& = \frac{1}{(N-1)!} \lim_{s \rightarrow 1/2+k} \sqbra{(s-1/2-k)^N \cdot \roundbra{\Gamma(x+s) \Gamma^N(1/2-s)} z^{-s}}^{(N-1)}
\\
& = \frac{1}{(N-1)!} \lim_{s \rightarrow 1/2+k} \sqbra{\roundbra{\roundbra{s-1/2-k}\Gamma\roundbra{1/2-s}}^N \cdot \Gamma(x+s) z^{-s}}^{(N-1)}
\\
& = \frac{1}{(N-1)!} \lim_{s \rightarrow 1/2+k} \sqbra{\mathcal{H}_1(s) \cdot {}_{x}\mathcal{H}_2(s) z^{-s}}^{(N-1)},% \label{eq:residualder}
\end{aligned}
\end{equation}
where
\begin{equation}
\begin{aligned}
\mathcal{H}_1(s) & \coloneqq \roundbra{\roundbra{s-1/2-k}\Gamma\roundbra{1/2-s}}^N  \\ %\label{eq:defH1}\\
{}_{x}\mathcal{H}_2(s) & \coloneqq \Gamma(x+s).
\end{aligned}
\end{equation}
Similarly as before, using the Leibniz rule we obtain
\begin{align}
\Res_{s=1/2+k} \sqbra{\mathcal{H}_{N+1,1}^{0,N+1}(s)z^{-s}}
%\nonumber
%\\
%& \quad
& = z^{-1/2-k} \sum_{j=0}^{N-1}  \left\{\frac{1}{(N-1)!} \sum_{n=j}^{N-1}  (-1)^j \binom{N-1}{n}  \binom{n}{j}  \right. 
\nonumber \\ \nonumber
& \quad \cdot \left. \lim_{s \rightarrow 1/2+k}  \sqbra{\mathcal{H}_1(s)}^{(N-1-n)}  \lim_{s \rightarrow 1/2+k} \sqbra{{}_{x}\mathcal{H}_2(s)}^{(n-j)}\right\} \sqbra{\log z}^j 
\\ % \label{eq:Resdevelop} \\
&  = z^{-1/2-k} \sum_{j=0}^{N-1}  {}_{x}H_{kj} \cdot \sqbra{\log z}^j,
\end{align}
where
\begin{equation}\label{eq:xHkj1}
H_{kj}^x \coloneqq \frac{1}{(N-1)!} \sum_{n=j}^{N-1}  (-1)^j \binom{N-1}{n}  \binom{n}{j} \lim_{s \rightarrow 1/2+k}  \sqbra{\mathcal{H}_1(s)}^{(N-1-n)}  \lim_{s \rightarrow 1/2+k} \sqbra{{}_{x}\mathcal{H}_2(s)}^{(n-j)}.
\end{equation}
For ${}_{x}\mathcal{H}_2(s)=\Gamma(x+s)$ with $x \in \{0,1\}$ and $\ell \in \N_0$ it holds that
\begin{align}
{}_{x}\mathcal{H}_2^{(\ell)}(s)&= \Gamma^{(\ell)}(x+s) 
\intertext{and}
\lim_{s \rightarrow 1/2+k} {}_{x}\mathcal{H}^{(\ell)}_2(s) &= \Gamma^{(\ell)}(1/2+x+k)\quad \text{for all } k \in \N_0.
\end{align}
By \eqref{eq:limH1N-1-n}, which is a consequence of Lemma \ref{lem:H1s},
\begin{equation}
\begin{aligned}
\lim_{s \rightarrow 1/2+k} & \sqbra{\mathcal{H}_1(s)}^{(N-1-n)} = (-1)^{Nk-n-1} (N-1-n)! 
\\
& \cdot \sum_{j_1+\ldots+j_N=N-1-n} \ \prod_{t=1}^N
    \curlybra{\sum_{\ell_1+\ldots+\ell_{k+1}=j_t} \frac{\Gamma^{(\ell_{k+1})}(1)}{\ell_{k+1}!} \curlybra{\prod_{i=1}^{k-1} \roundbra{k-i+1}^{-(\ell_i+1)}}}.
\end{aligned}
\end{equation}
Plugging this into \eqref{eq:xHkj1} and simplifying leads to
\begin{equation}
\begin{aligned}
H_{kj}^x &= \frac{(-1)^{Nk-1}}{j!}  \sum_{n=j}^{N-1}  \frac{(-1)^{n-j}}{(n-j)!} \Gamma^{(n-j)}(1/2+x+k) \\
& \quad \quad \cdot  \sum_{j_1+\ldots+j_N=N-1-n} \ \prod_{t=1}^N
    \curlybra{\sum_{\ell_1+\ldots+\ell_{k+1}=j_t} \frac{\Gamma^{(\ell_{k+1})}(1)}{\ell_{k+1}!} \curlybra{\prod_{i=1}^{k-1} \roundbra{k-i+1}^{-(\ell_i+1)}}} ,
\end{aligned}
\end{equation}
which finishes the proof.
\end{proof}

Finally, we provide a proof of Lemma~\ref{lemma:sumZpowerlog}. The proof is analogous to the one of the previous lemma (Lemma~\ref{lemma:sumYpowerlog}).
\begin{proof}[Proof of Lemma~\ref{lemma:sumZpowerlog}]\hfill\\
By~\cite[Theorem 1.2 (iii)]{Htransform}, the Meijer G-function $G_{N,2}^{2,N} \left(z \Big| \begin{matrix}  1/2, 1/2, \ldots, 1/2 \\ 1/2, x \end{matrix} \right)$ with $x \in \{0,1\}$ is an analytic function of $z$ in the sector $\abs{\arg z}<\frac{(N+2)\pi}{2}$ (this is true for $N \geq 3$ since $\arg z=0$ for $z \in \R_0^+$ and $\arg z=\pi$, for $z \in \R^-$).
Therefore, we have
\begin{equation}
G_{N,2}^{2,N} \left(z \Big| \begin{matrix} 1/2, 1/2, \ldots, 1/2 \\ 1/2, x \end{matrix} \right)= -\sum_{i=1}^{N}\sum_{k=0}^{\infty} \Res_{s=a_{ik}} \sqbra{{}_{x}\mathcal{H}_{N,2}^{2,N}(s) z^{-s}},
\end{equation}
where ${}_{x}\mathcal{H}_{N,2}^{2,N}$ is given by \eqref{Hmnpq} with $m=q=2$, $n=p=N$,  $a_1=\ldots=a_{N}=1/2$, and $b_1=1/2$, $b_2=x$, i.e.\
\begin{align}
{}_{x}\mathcal{H}_{N,2}^{2,N} (s) & =  \Gamma(1/2+s) \Gamma(x+s) \Gamma^N(1/2-s)
\end{align}
and $a_{ik}=1-a_i+k$ with $k \in \NN_0$ denote the poles of $\Gamma(1-a_i-s)$.
In our scenario we have only poles $a_{1k}=1/2+k$ of order $N$ with $k \in \N_0$.  Therefore,
\begin{equation}\label{MeijerGinResZ}
G_{N,2}^{2,N} \left(z \Big| \begin{matrix} 1/2, 1/2, \ldots 1/2 \\ 1/2, x \end{matrix} \right)=  \sum_{k=0}^{\infty} -\Res_{s=a_{1k}} \sqbra{{}_{x}\mathcal{H}_{N,2}^{2,N}(s) z^{-s}}.
\end{equation}
Since $a_{1k}=1/2+k$ is an $N$-th order pole of $\Gamma(1/2-s)$ for all $k \in \N_0$, also the integrand ${}_{x}\mathcal{H}_{N,2}^{2,N}(s)= \Gamma(1/2+s) \Gamma(x+s) \Gamma^N(1/2-s)$ has a pole $a_{1k}$ of order $N$ for all $k \in \N_0$.
Thus
\begin{align}
& \Res_{s=1/2+k} \sqbra{{}_{x}\mathcal{H}_{N,2}^{2,N}(s)z^{-s}} \nonumber \\
& = \frac{1}{(N-1)!} \lim_{s \rightarrow 1/2+k} \sqbra{(s-1/2-k)^N {}_{x}\mathcal{H}_{N,2}^{2,N}(s) z^{-s}}^{(N-1)}
\nonumber
\\
& = \frac{1}{(N-1)!} \lim_{s \rightarrow 1/2+k} \sqbra{(s-1/2-k)^N \cdot \roundbra{\Gamma(1/2+s) \Gamma(x+s) \Gamma^N(1/2-s)} z^{-s}}^{(N-1)}
\nonumber
\\
& = \frac{1}{(N-1)!} \lim_{s \rightarrow 1/2+k} \sqbra{\roundbra{\roundbra{s-1/2-k}\Gamma\roundbra{1/2-s}}^N \cdot \Gamma(1/2+s) \Gamma(x+s) z^{-s}}^{(N-1)}
\nonumber
\\
& = \frac{1}{(N-1)!} \lim_{s \rightarrow 1/2+k} \sqbra{\mathcal{H}_1(s) \cdot {}_{x}\bar{\mathcal{H}}_2(s) z^{-s}}^{(N-1)}, \label{eq:residualderZ}
\end{align}
where
\begin{equation}
\begin{aligned}
\mathcal{H}_1(s) & \coloneqq \roundbra{\roundbra{s-1/2-k}\Gamma\roundbra{1/2-s}}^N \label{eq:defH1Z}\\
{}_{x}\bar{\mathcal{H}}_2(s) & \coloneqq \Gamma(1/2+s) \Gamma(x+s).
\end{aligned}
\end{equation}
Similarly as before, using the Leibniz rule we obtain
\begin{equation}
\begin{aligned}
&\Res_{s=1/2+k} \sqbra{\mathcal{H}_{N+1,1}^{0,N+1}(s)z^{-s}}
\\
& \quad = z^{-1/2-k} \sum_{j=0}^{N-1}  \left\{\frac{1}{(N-1)!} \sum_{n=j}^{N-1}  (-1)^j \binom{N-1}{n}  \binom{n}{j}  \right. \\
& \quad \quad \quad  \quad
\left. \cdot \lim_{s \rightarrow 1/2+k}  \sqbra{\mathcal{H}_1(s)}^{(N-1-n)}  \lim_{s \rightarrow 1/2+k} \sqbra{{}_{x}\bar{\mathcal{H}}_2(s)}^{(n-j)}\right\} \sqbra{\log z}^j \\ %\label{eq:Resdevelop} \\
& \quad = z^{-1/2-k} \sum_{j=0}^{N-1}  \bar{H}_{kj}^x \cdot \sqbra{\log z}^j,
\end{aligned}
\end{equation}
where
\begin{equation}\label{eq:xHkj1Z}
\bar{H}_{kj}^x \coloneqq \frac{(-1)^j}{(N-1)!} \sum_{n=j}^{N-1}   \binom{N-1}{n}  \binom{n}{j} \lim_{s \rightarrow 1/2+k}  \sqbra{\mathcal{H}_1(s)}^{(N-1-n)}  \lim_{s \rightarrow 1/2+k} \sqbra{{}_{x}\bar{\mathcal{H}}_2(s)}^{(n-j)}.
\end{equation}
For ${}_{x}\bar{\mathcal{H}}_2(s)=\Gamma(1/2+s)\Gamma(x+s)$ with $x \in \{0,1\}$ and $\ell \in \N_0$ it holds that
\begin{equation}
{}_{x}\bar{\mathcal{H}}_2^{(\ell)}(s) = \sum_{i=0}^{\ell} \binom{\ell}{i}\Gamma^{(\ell-i)}(1/2+s)\Gamma^{(i)}(x+s)
\end{equation}
 and
\begin{equation}
 \lim_{s \rightarrow 1/2+k} {}_{x}\bar{\mathcal{H}}^{(\ell)}_2(s) = \sum_{i=0}^{\ell} \binom{\ell}{i}\Gamma^{(\ell-i)}(1+k)\Gamma^{(i)}(1/2+x+k)\quad \text{for all } k \in \N_0.
\end{equation}
By \eqref{eq:limH1N-1-n}, which is a consequence of Lemma \ref{lem:H1s},
\begin{equation}
\begin{aligned}
\lim_{s \rightarrow 1/2+k} &\sqbra{\mathcal{H}_1(s)}^{(N-1-n)}  = (-1)^{Nk-n-1} (N-1-n)! \\
& \cdot \sum_{j_1+\ldots+j_N=N-1-n} \ \prod_{t=1}^N
    \curlybra{\sum_{\ell_1+\ldots+\ell_{k+1}=j_t} \frac{\Gamma^{(\ell_{k+1})}(1)}{\ell_{k+1}!} \curlybra{\prod_{i=1}^{k-1} \roundbra{k-i+1}^{-(\ell_i+1)}}} .
\end{aligned}
\end{equation}
Plugging this into \eqref{eq:xHkj1Z} and simplifying leads to
\begin{equation}
\begin{aligned}
\bar{H}_{kj}^x &= \frac{(-1)^{Nk-1}}{j!}  \sum_{n=j}^{N-1}  \frac{(-1)^{n-j}}{(n-j)!} \roundbra{\sum_{i=0}^{n-j} \binom{n-j}{i}\Gamma^{(n-j-i)}(1+k)\Gamma^{(i)}(1/2+x+k)} \\
& \quad \quad \cdot  \sum_{j_1+\ldots+j_N=N-1-n} \ \prod_{t=1}^N
    \curlybra{\sum_{\ell_1+\ldots+\ell_{k+1}=j_t} \frac{\Gamma^{(\ell_{k+1})}(1)}{\ell_{k+1}!} \curlybra{\prod_{i=1}^{k-1} \roundbra{k-i+1}^{-(\ell_i+1)}}},
\end{aligned}
\end{equation}
which finishes the proof.
\end{proof}

\end{document}